\documentclass[11pt]{article}
\usepackage{graphicx}
\usepackage{tabularx}
\usepackage{url}
\usepackage{amsthm}

\usepackage{color}
\usepackage{xspace}
\usepackage{tikz}
\usetikzlibrary{decorations,decorations.pathmorphing,decorations.pathreplacing,fit,positioning,arrows}

\tikzstyle{w_vertex}=[circle,fill=black!100,text=white,inner sep=0.4mm,draw]
\tikzstyle{vertex}=[circle,fill=black!100,text=white,inner sep=0.8mm]
\tikzstyle{point}=[circle,fill=black,inner sep=0.1mm]

\theoremstyle{plain}
\newtheorem{theorem}{Theorem}
\newtheorem{lemma}{Lemma}

\newtheorem*{claim*}{Claim}

\newtheorem{corollary}{Corollary}
\theoremstyle{definition}

\newtheorem{conjecture}{Conjecture}

\theoremstyle{remark}

\date{}
\title{Graph parameters, implicit representations \\ and factorial properties\thanks{Some results presented in this paper appeared in the extended abstract 
\cite{IWOCA2022} published in the proceedings of the 33rd International Workshop on Combinatorial Algorithms, IWOCA 2022.}}
\author{ B. Alecu\thanks{School of Computing, University of Leeds, UK. Email:  B.Alecu@leeds.ac.uk}  
\and V.E. Alekseev 
\and A. Atminas\thanks{Department of Mathematical Sciences, Xi'an Jiaotong-Liverpool University, 111 Ren'ai Road, Suzhou 215123, China. Email: Aistis.Atminas@xjtlu.edu.cn. Aistis Atminas was supported by XJTLU Research Development Fund RDF-22-01-070} 
\and V. Lozin\thanks{Mathematics Institute, University of Warwick, Coventry, CV4 7AL, UK. Email:  V.Lozin@warwick.ac.uk} 
\and V. Zamaraev\thanks{Department of Computer Science, University of Liverpool, Ashton Street, Liverpool, L69 3BX, UK. Email: Viktor.Zamaraev@liverpool.ac.uk}}

\usepackage{caption}
\usepackage{amsmath} 
\usepackage{amsfonts}
\usepackage[cp1251]{inputenc}
\usepackage{colordvi}
\usepackage{graphicx}
\usepackage{tikz}
\usepackage{subcaption}
\usetikzlibrary{decorations.pathreplacing,arrows,automata}
\DeclareMathOperator{\sd}{sd}

\DeclareMathOperator{\dd}{dd}
\DeclareMathOperator{\cont}{cont}

\begin{document}
\maketitle

\newtheorem{obs}{Observation}

\newtheorem{prop}{Proposition}
\newtheorem{cor}{Corollary}

\def\N{\mathbb{N}}
\def\t{\sim}
\def\nt{\nsim}
\def\1{n+1}
\def\2{n+2}

\begin{abstract} How to efficiently represent a graph in computer memory is a fundamental data structuring question. 
In the present paper, we address this question from a combinatorial point of view. 
A representation of an $n$-vertex graph $G$ is called implicit if it assigns to each vertex of $G$ a binary code of length $O(\log n)$ so that the adjacency of two
vertices is a function of their codes. A necessary condition for a hereditary class $\mathcal X$ of graphs to admit an implicit representation is that $\mathcal X$ has 
at most factorial speed of growth. This condition, however, is not sufficient, as was recently shown in
[Hatami \& Hatami, FOCS 2022]. 
Several sufficient conditions for the existence of implicit representations deal with boundedness of some parameters, such as degeneracy or clique-width. 
In the present paper, we analyze more graph parameters and prove a number of new results related to implicit representation and factorial properties.  
\end{abstract}

{\it Keywords}: Graph parameter; Implicit representation; Hereditary class; Factorial property

%%%%%%%%%%%%%%%%%%%%%%%%%%%%%%%%%%%%%%%%%%%%

\section{Introduction}

Every simple graph with $n$ vertices can be represented by a binary word of length $\binom{n}{2}$ (one bit per pair of vertices), 
and if no \emph{a priori} information about the graph is known, this representation is optimal. 
However, for graphs belonging to certain classes, this representation can be substantially shortened.
For instance, the Pr\"ufer code allows representing a labelled tree with $n$ vertices by a binary word of length $n\log n$. 
This is optimal among all representations of a labelled graph, because we need $\log n$ bits for each vertex just to represent its label. 
Of course, those $n\log n$ bits describing the vertex labels do not in general describe the graph itself, 
since these bits do not necessarily allow us to compute the adjacencies.
However, it is sometimes possible to represent not only the labels, but also the edges, using only $O(\log n)$ bits per vertex. 
If additionally the adjacency between two vertices can be computed from their codes (i.e.\ from their labels of length $O(\log n)$), 
then the graph is said to be represented {\it implicitly}.

The idea of implicit representation was introduced in \cite{implicit-1}. Its importance is due to the following reasons.
First, it is order-optimal, i.e.\ within a factor of the optimal representation. Second, it allows one to store information about graphs locally, which is crucial in distributed computing. 
Finally, it is applicable to graphs in various classes of practical or theoretical importance, such as graphs of bounded vertex degree, of bounded clique-width, 
planar graphs, interval graphs, permutation graphs, line graphs, etc. 

To better describe the area of applicability of implicit representations, let us observe that if graphs in a class $\mathcal X$ 
admit an implicit representation, then the number of $n$-vertex labelled graphs in $\mathcal X$, also known as the {\it speed} of $\mathcal X$, 
must be $2^{O(n\log n)}$, since the number of graphs cannot be larger than the number of binary words representing them. 
In the terminology of \cite{SpHerProp}, hereditary classes containing  $2^{\Theta(n\log n)}$ $n$-vertex labelled graphs have {\it factorial} speed of growth.
The family of factorial classes, i.e.\ hereditary classes with a factorial speed of growth, is rich and diverse. In particular, 
it contains all classes mentioned earlier and a variety of other classes, such as unit disk graphs, classes of graphs of bounded arboricity, of bounded functionality \cite{functionality}, etc.
The authors of \cite{implicit-1}, who introduced the notion of implicit representation, asked whether {\it every} hereditary class of speed 
$2^{O(n\log n)}$ admits such a representation. 

Recently, Hatami and Hatami \cite{false} answered this question negatively by proving the existence of a factorial class of bipartite graphs
that does not admit an implicit representation. This negative result raises the following natural question:
if the speed is not responsible for implicit representation, then what is responsible for it?

Looking for an answer to this question, we observe that most positive results on implicit representations deal with classes where certain graph parameters, such as degeneracy or twin-width \cite{twin},
are bounded. In an attempt to produce more positive results, in this paper we analyze more graph parameters.
These parameters are of interest on their own right, and in Section~\ref{sec:parameters} we prove a number of results related to them. 
Some of these results suggest new candidate classes for implicit representation, and in Section~\ref{sec:implicit} we develop such representations for them.
  
Finally, we observe that, in spite of the negative result in \cite{false}, factorial speed remains a necessary condition 
for an implicit representation in a hereditary class $\mathcal X$, and determining the speed of $\mathcal X$ is the first natural step towards 
deciding whether such a representation exists. Some results on this topic are presented in Section~\ref{sec:factorial}.

All relevant preliminary information can be found in Section~\ref{sec:pre}. 
%Section~\ref{sec:parameters} is devoted to graph parameters, 
%Section~\ref{sec:implicit} deals with implicit representations, and Section~\ref{sec:factorial} with factorial properties. 
Section~\ref{sec:con} concludes the paper with a number of open problems.

%%%%%%%%%%%%%%%%%%%%%%%%%%%%%%%%%%%%%%%%%%%%%%%%%%%%%%%%%%%%%%%%%%%%%%%%%%%%
%%%%%%%%%%%%%%%%%%%%%%%%%%%%%%%%%%%%%%%%%%%%%%%%%%%%%%%%%%%%%%%%%%%%%%%%%%%%

\section{Preliminaries}
\label{sec:pre}
%%%%%%%%%%%%%%%%%%%%%%%%%%%%%%%%%%%%%%%%%%%%%%%%%%%%%%%%%%%%%%%%%%%%%%%%%%%%
%%%%%%%%%%%%%%%%%%%%%%%%%%%%%%%%%%%%%%%%%%%%%%%%%%%%%%%%%%%%%%%%%%%%%%%%%%%%
All graphs in this paper are simple, i.e.\ undirected, without loops or multiple edges. 
The vertex set and the edge set of a graph $G$ are denoted $V(G)$ and $E(G)$, respectively. 
The neighbourhood of a vertex $x\in V(G)$, denoted $N(x)$, is the set of vertices adjacent to $x$,
and the degree of $x$, denoted $\deg(x)$, is the size of its neighbourhood. The codegree of $x$ is 
the number of vertices non-adjacent to $x$. By $[n]$ we denote the set of integers between $1$ and $n$ inclusive.

As usual, $K_n,P_n$ and $C_n$ denote a complete graph, a chordless path and a chordless cycle on $n$ vertices, respectively.  
$S_{i,j,k}$ is a tree with exactly three leaves with distances $i,j,k$ from the only vertex of degree 3. 
By $nG$ we denote the disjoint union of $n$ copies of $G$.

The subgraph of $G$ induced by a set $U\subseteq V(G)$ is denoted $G[U]$. 
If $G$ does not contain an induced subgraph isomorphic to a graph $H$, we say that $G$ is $H$-free, or that $G$ excludes $H$,
or that $H$ is a forbidden induced subgraph for $G$.  

In a graph, a {\it clique} is a subset of pairwise adjacent vertices and an {\it independent set} is a subset of pairwise non-adjacent vertices.
A {\em homogeneous set} is a subset of vertices, which is either a clique or an independent set.

A graph $G=(V,E)$ is {\it bipartite} if its vertex set can be partitioned into two independent sets. 
A bipartite graph given together with a bipartition of its vertex set into two independent sets $A$ and $B$ will be denoted $G=(A,B,E)$.
The {\it bipartite complement} of a bipartite graph $G=(A,B,E)$ is the bipartite graph $\widetilde{G}:=(A,B,(A\times B)-E)$.
The {\it bi-codegree} of a vertex $x$ in a bipartite graph $G=(A,B,E)$ is the degree of $x$ in $\widetilde{G}$.
By $K_{n,m}$ we denote a complete bipartite graph with parts of size $n$ and $m$. The graph $K_{1,n}$ (for some $n$) is called a {\em star}.
A \emph{chain graph} is a bipartite graph whose vertices in one of the parts can be linearly ordered with respect to the inclusion of their neighbourhoods. The class of chain graphs is precisely the class of $2K_2$-free bipartite graphs.

Given two bipartite graphs $G_1=(A_1,B_1,E_1)$ and $G_2=(A_2,B_2,E_2)$, we say that $G_1$ does not contain 
a {\it one-sided copy} of $G_2$ if there is no induced copy of $G_2$ in $G_1$ with $A_2\subseteq A_1$ or 
there is no induced copy of $G_2$ in $G_1$ with $A_2\subseteq B_1$.

We say that a graph $G$ is {\it co-bipartite} if it is the complement of a bipartite graph, and that $G$ is {\it split} 
if the vertex set of $G$ can be partitioned into a clique and an independent set.  

\subsection{Graph classes}

A class of graphs is {\it hereditary} if it is closed under taking induced subgraphs. 
It is well known that a class $\mathcal X$ is hereditary if and only if $\mathcal X$ can be described by a set of minimal forbidden induced subgraphs. 
In this section, we introduce a few hereditary classes that play an important role in this paper. 

%Motivated by the negative result in \cite{false}, we focus on subclasses of bipartite graphs.
%One of them is the class of {\it chain graphs}, which is precisely the class of $2K_2$-free bipartite graphs. 

Most of our results deal with hereditary classes of bipartite graphs, which is motivated by the negative result in \cite{false} and the following argument.
A natural way to transform any graph into a bipartite graph is to interpret its adjacency matrix as a bipartite
adjacency matrix. This extends to a transformation between hereditary classes: transform every graph in a hereditary class to a bipartite graph and take the hereditary closure of the obtained set of bipartite graphs. As was shown in \cite{HWZ21}, this transformation preserves the factorial speed of growth as well as the existence of an implicit representation\footnote{In \cite{HWZ21}, the transformation was shown to preserve a specific type of implicit representations, but the argument works for arbitrary implicit representations.}.
%This motivates the study of bipartite graphs and most of our results are about hereditary classes of bipartite graphs.

\paragraph{Monogenic classes of bipartite graphs.}

Even in the case of \emph{monogenic} classes of bipartite graphs, i.e. classes of bipartite graphs defined by a single forbidden induced bipartite subgraph, characterizing which classes are factorial is not straightforward.
In \cite{Allen}, Allen identified nearly all factorial classes in this family, with the exception of $P_7$-free bipartite graphs. 
This exceptional class was characterised as factorial in \cite{P7}, which leads to a dichotomy presented in Theorem~\ref{thm:speed-dichotomy} below.
This theorem follows readily from the results in \cite{Allen} and \cite{P7} and Lemma~\ref{lem:3-graphs}.
%This lemma has never been formally proved, although implicitly it appears in \cite{Allen}.
%For completeness, we present Lemma~\ref{lem:3-graphs} together with a proof. 
The graph $F_{t,p}$ mentioned in the lemma is presented in Figure~\ref{fig:speed}. 

\begin{figure}[ht]
\begin{center} 
\begin{picture}(100,40)
\setlength{\unitlength}{0.4mm}
\put(0,0){\circle*{5}}
\put(10,0){\circle*{5}}
\put(30,0){\circle*{5}}
\put(40,0){\circle*{5}}
\put(20,42){\circle{5}}
\put(15,0){\circle*{1}}
\put(20,0){\circle*{1}}
\put(25,0){\circle*{1}}
\put(20,40){\line(-1,-2){20}}
\put(20,40){\line(-1,-4){10}}
\put(20,40){\line(1,-2){20}}
\put(20,40){\line(1,-4){10}}
%\put(105,1){(a)}

\put(90,0){\circle*{5}}

\put(50,0){\circle*{5}}
\put(70,0){\circle*{5}}
\put(80,0){\circle*{5}}
\put(60,42){\circle{5}}
\put(55,0){\circle*{1}}
\put(60,0){\circle*{1}}
\put(65,0){\circle*{1}}
\put(60,40){\line(-1,-2){20}}
\put(60,40){\line(-1,-4){10}}
\put(60,40){\line(1,-2){20}}
\put(60,40){\line(1,-4){10}}

\put(-4,-8){$x_1$}
\put(11,-8){$\ldots\ x_t$}
\put(48,-8){$y_1$}
\put(61,-8){$\ldots \  y_p$}

\end{picture} 
\end{center}
\caption{The graph $F_{t,p}$ }
\label{fig:speed}
\end{figure}
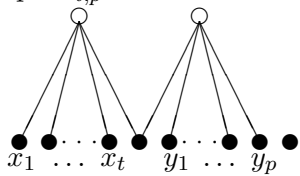

\begin{lemma}\label{lem:3-graphs}
A bipartite graph $H$ is simultaneously a forest and the bipartite complement of a forest if and only if $H$ is an induced subgraph of a $P_7$, of an $S_{1,2,3}$ or of a graph $F_{t,p}$.
\end{lemma} 

\begin{proof}
The ``if'' part of the proof is obvious. To prove the ``only if'' part, asssume a bipartite graph $H=(A,B,E)$ is simultaneously a forest and the bipartite complement of a forest.
Then any two vertices of the same colour in $H$ share at most one neighbour and at most one non-neighbour. Therefore, if one of the parts of $H$ contains at most two 
vertices then $H$ is an  induced subgraph of a graph $F_{t,p}$ for some $t$ and $p$. 
From now on, we assume that each part of $H$ contains at least three vertices, say $a_1,a_2,a_3\in A$ and $b_1,b_2,b_3\in B$. 

Suppose first that neither $H$ nor $\widetilde{H}$ contain vertices of degree more than 2. Then every connected component of $H$ (and of   $\widetilde{H}$) is a path. 
None of the components of $H$ is trivial (a singleton), since otherwise the opposite part has size at most two, contradicting our assumption. Also, the number of non-trivial components 
is at most 2, since $3P_2=\widetilde{C}_6$. If $H$ is connected, then $H=P_k$ with $k\le 7$, since $\widetilde{P}_8$ contains a $C_4$. If $H$ consists of two components $P_k$ and $P_t$
($k\ge t\ge 2$), then $k,t\le 4$, since otherwise an induced $3P_2=\widetilde{C}_6$ arises. Moreover, if $k=4$, then $t=2$, since otherwise $\widetilde{H}$ contains a vertex of degree at least $3$.
It follows that $H$ is an induced subgraph of $P_7$. If $k\le 3$, then  $H$ is again an induced subgraph of $P_7$.

We are left with the case in which $H$ (or $\widetilde{H}$ -- the cases are symmetric) contains a vertex of degree at least $3$. 
Say this vertex is $a_1$, and that it is adjacent to $b_1,b_2,b_3$ in $H$. Then
each of $a_2$ and $a_3$ has exactly one neighbour in $\{b_1,b_2,b_3\}$ and these neighbours are different, since otherwise an induced $C_4$ or an induced  $\widetilde{C}_4$ arises.
Without loss of generality, let $a_2$ be adjacent to $b_2$ and let $a_3$ be adjacent to $b_3$. 

Assume $A$ contains at least one more vertex, say $a_4$.
Then by the same arguments $a_4$ has exactly one neighbour in $\{b_1,b_2,b_3\}$ and this neighbour must be different from $b_2$ and $b_3$, i.e.\ $a_4$ is adjacent to $b_1$.
But then $a_2,a_3,a_4,b_1,b_2,b_3$ induce $3P_2=\widetilde{C}_6$. This contradiction shows that $A=\{a_1,a_2,a_3\}$. 

Assume $B$ contains at least one more vertex, say $b_4$. Then $b_4$ is not adjacent to $a_1$, since otherwise an induced $C_4$ or an induced  $\widetilde{C}_4$ arises.
Additionally, $b_4$ is adjacent to exactly one of $a_2$ and $a_3$, since otherwise an induced $C_6$ or an induced  $\widetilde{C}_4$ arises.
Without loss of generality, suppose $b_4$ is adjacent to $a_2$. If $B$ contains one more vertex, say $b_5$, then by the same arguments, $b_5$ is adjacent to exactly one of $a_2$ and $a_3$.
If $b_5$ is adjacent to $a_2$, then $a_1,a_3,b_4,b_5$ induce a  $\widetilde{C}_4$, and if $b_5$ is adjacent to $a_3$, then $a_1,a_2,a_3,b_1,b_4,b_5$ induce $3P_2=\widetilde{C}_6$.
A contradiction in both cases shows $B=\{b_1,b_2,b_3,b_4\}$ and hence $H$ is an induced $S_{1,2,3}$. 
\end{proof}

\begin{theorem}\label{thm:speed-dichotomy}
For  a bipartite graph $H$, the class of $H$-free bipartite graphs has at most factorial speed of growth if and only if 
$H$ is an induced subgraph of one of the following graphs: $P_7$, $S_{1,2,3}$ and $F_{p,q}$.  
\end{theorem}

\begin{proof}
The factorial speed of $F_{t,p}$-free bipartite graphs and $S_{1,2,3}$-free bipartite graphs was shown in \cite{Allen}, while for $P_7$-free bipartite graphs, it was shown in
\cite{P7}. If $H$ is not an induced subgraph of $P_7$, $S_{1,2,3}$ or $F_{t,p}$, then by Lemma~\ref{lem:3-graphs}, 
either $H$ or $\widetilde{H}$ contains a cycle. It follows that the speed of the class of $H$-free bipartite graphs is superfactorial (this is a well-known fact that can also be found in \cite{Allen}).
\end{proof}

A similar dichotomy of monogenic classes of bipartite graphs with respect to the existence of an implicit representation is not known yet. Since at most factorial speed of growth is a necessary condition for a class to admit an implicit representation, it readily follows from Theorem \ref{thm:speed-dichotomy} that if a class of $H$-free bipartite graphs admits an implicit representation then $H$ is an induced subgraph of $P_7$, $S_{1,2,3}$, or $F_{p,q}$.  
It is known that the class of $S_{1,2,3}$-free bipartite graphs admits an implicit representation because this class has bounded clique-width \cite{skew-star} and graph classes of bounded clique-width admit an implicit representation \cite{Spinrad}. Prior to this work, the question remained open for the other two cases.
In Section \ref{sec:Ftp} we resolve the case of $F_{p,q}$-free bipartite graphs, by showing that
any such class admits an implicit representation.

\paragraph{Chordal bipartite graphs.} 
A bipartite graph is {\it chordal} bipartite if it has no chordless (induced) cycles of length at least $6$. 
The class of $C_4$-free chordal bipartite graphs is precisely the class of forests, which is factorial and admits an implicit representation.
Despite this `closeness' to the class of forests, the class of chordal bipartite graphs is superfactorial \cite{Spi95} and hence does not admit an implicit representation. This makes the class of chordal bipartite graphs a natural area for the study of factorial and implicitly representable graph classes.
 
It is known that classes of $K_{p,q}$-free chordal bipartite graphs have bounded tree-width \cite{chordal-factorial}, and hence admit an implicit representation. Some other subclasses of chordal
bipartite graphs are not known to admit an implicit representation, but they are known to be factorial.
These include classes of chordal bipartite graphs excluding a fixed forest \cite{chordal-bipartite}.
In the present paper we
reveal a number of new factorial subclasses of chordal bipartite graphs and show that some of them admit an implicit representation. Among other results, we show that the class of $S_{2,2,2}$-free chordal bipartite graphs, and any class of chordal bipartite graphs avoiding a fixed chain graph admit an implicit representation.

%Another important area for investigation of factorial graph classes
%is the class of chordal bipartite graphs. 

%%%%%%%%%%%%%%%%%%%%%%%%%%%%%%%%%%%%%%%%%%%%%
\subsection{Tools}
%%%%%%%%%%%%%%%%%%%%%%%%%%%%%%%%%%%%%%%%%%%%%

Several useful tools to produce an implicit representation have been introduced in \cite{implicit}. 
In this section, we mention two such tools, and generalise one of them.

The first result deals with the notion of locally bounded coverings, which can be defined as follows.
Let $G$ be a graph. A set of graphs $H_1,\ldots,H_k$ is called a {\it covering} of $G$ if the union 
of $H_1,\ldots,H_k$ coincides with $G$, i.e.\ if $V(G)=\bigcup\limits_{i=1}^k V(H_i)$  and 
$E(G)=\bigcup\limits_{i=1}^k E(H_i)$.

\begin{theorem}{\rm \cite{implicit}}\label{thm:local}
Let $\mathcal X$ be a class of graphs and $c$ a constant. If every graph $G\in \mathcal X$ can be covered by graphs from 
a class $\mathcal Y$ admitting an implicit representation in such a way that every vertex of $G$ is covered by at most $c$ graphs, 
then $\mathcal X$ also admits an implicit representation.
\end{theorem}

The second result deals with the notion of partial coverings and can be stated as follows. 

\begin{theorem}{\rm \cite{implicit}}\label{thm:partial}
Let $\mathcal X$  be a hereditary class. Suppose there is a constant $d$ and a hereditary class $\mathcal Y$  which admits an implicit representation
such that every graph $G\in \mathcal X$ contains a non-empty subset $A\subseteq V(G)$ with the properties that $G[A]\in \mathcal Y$ and each vertex of $A$ 
has at most $d$ neighbours or at most $d$ non-neighbours in $V(G)-A$. Then $\mathcal X$ admits an implicit representation.
\end{theorem}

Next we provide a generalisation of Theorem~\ref{thm:partial} that will be useful later.

\begin{theorem}\label{thm:partial-general}
	Let $\mathcal X$ be a hereditary class. Suppose there is a constant $d$ and a hereditary class $\mathcal Y$ which admits an implicit representation so that every graph $G\in \mathcal X$ contains a non-empty subset $A\subseteq V(G)$ with the following properties:
	\begin{enumerate} 
		\item[(1)] $G[A]\in \mathcal Y$, 
		\item[(2)] $V(G)-A$ can be split into two subsets $B_1$ and $B_2$ with no edges between them, and
		\item[(3)] every vertex of $A$ has at most $d$ neighbours or at most $d$ non-neighbours  in $B_1$ and at most $d$ neighbours or at most $d$ non-neighbours in $B_2$.
	\end{enumerate}
	Then $\mathcal X$ admits an implicit representation.
\end{theorem}
\begin{proof}
	Let $G$ be an $n$-vertex graph in $\mathcal X$.
	We assign to the vertices of $G$ pairwise distinct \emph{indices} recursively as follows.
	Let $\{ 1, 2, \ldots, n \}$ be the \emph{index range} of $G$, and let $A$, $B_1$, and $B_2$
	be the partition of $V(G)$ satisfying the conditions (1)-(3) of the theorem.
	We assign to the vertices in $A$ indices from the interval $\{ |B_1|+1, |B_1|+2, \ldots, n-|B_2| \}$
	bijectively in an arbitrary way.
	We define the indices of the vertices in $B_1$ recursively by decomposing $G[B_1]$ and using the interval
	$\{ 1, 2, \ldots, |B_1| \}$ as its index range. Similarly, we define the indices of the vertices in $B_2$
	by decomposing $G[B_2]$ and using the interval $\{ n-|B_2|+1, n-|B_2|+2, \ldots, n \}$ as its index range.
	
	Now, for every vertex $v \in A$ its label consists of six components:
	\begin{enumerate}
		\item the label of $v$ in the implicit representation of $G[A] \in \mathcal Y$;
		\item the index of $v$;
		\item the index range of $B_1$, which we call the \emph{left index range} of $v$;
		\item the index range of $B_2$, which we call the \emph{right index range} of $v$;
		\item a boolean flag indicating whether $v$ has at most $d$ neighbours or $d$ non-neighbours in $B_1$ and the indices of those at most $d$ vertices;
		\item a boolean flag indicating whether $v$ has at most $d$ neighbours or $d$ non-neighbours in $B_2$ and the indices of those at most $d$ vertices.
	\end{enumerate}
	For the third and the fourth component we store only the first and the last elements of the ranges,
	and therefore the total label size is $O(\log n)$.
	The labels of the vertices in $B_1$ and $B_2$ are defined recursively.
	
	Note that two vertices can only be adjacent if either they have the same left and right index ranges 
	or the index of one of the vertices is contained in the left or right index range of the other vertex.
	In the former case, the adjacency of the vertices is determined by the labels in the first components
	of their labels. In the latter case, the adjacency is determined using the information stored in the components 5 and 6 of the labels.
\end{proof}

% In the context of bipartite graphs, this result can be restated as follows. 

% \begin{theorem}\label{thm:partial-bip}
% Let $X$  be a hereditary class of bipartite graphs. If there is a constant $d$ and a hereditary class $Y$  which admits an implicit representation
% such that every graph $G\in X$ contains a non-empty subset $A\subset V(G)$ such that $G[A]\in Y$ and each vertex of $A$ 
% has at most $d$ neighbours or at most $d$ non-neighbours in the \emph{opposite} part outside of $A$, then $X$ admits an implicit representation.
% \end{theorem}

In the context of bipartite graphs, Theorem \ref{thm:partial-general} can be adapted as follows.

\begin{theorem}\label{thm:partial-bip}
	Let $\mathcal X$ be a hereditary class of bipartite graphs. Suppose there is a constant $d$ and a hereditary class $\mathcal Y$ which admits an implicit representation so that every graph $G\in \mathcal X$ contains a non-empty subset $A\subseteq V(G)$ with the following properties:
	\begin{enumerate} 
		\item[(1)] $G[A]\in \mathcal Y$, 
		\item[(2)] $V(G)-A$ can be split into two subsets $B_1$ and $B_2$ with no edges between them, and
		\item[(3)] every vertex $v$ of $A$ has at most $d$ neighbours or at most $d$ non-neighbours in the  part of $B_1$ which is \emph{opposite} to the part of $A$ containing $v$, 
and at most $d$ neighbours or at most $d$ non-neighbours in the  part of $B_2$  which is \emph{opposite} to the part of $A$ containing $v$.
	\end{enumerate}
	Then $\mathcal X$ admits an implicit representation.
\end{theorem}

%%%%%%%%%%%%%%%%%%%%%%%%%%%%%%%%%%%%%%%%%%%%%%%%%%%%%%%%%%%%%%%%%%%%%%%%%%%%
%%%%%%%%%%%%%%%%%%%%%%%%%%%%%%%%%%%%%%%%%%%%%%%%%%%%%%%%%%%%%%%%%%%%%%%%%%%%

\section{Graph parameters}
\label{sec:parameters}

%%%%%%%%%%%%%%%%%%%%%%%%%%%%%%%%%%%%%%%%%%%%%%%%%%%%%%%%%%%%%%%%%%%%%%%%%%%%
%%%%%%%%%%%%%%%%%%%%%%%%%%%%%%%%%%%%%%%%%%%%%%%%%%%%%%%%%%%%%%%%%%%%%%%%%%%%

It is easy to see that classes of bounded vertex degree admit an implicit representation.
More generally, bounded degeneracy in a class provides us with an implicit representation,
where the {\it degeneracy} of a graph $G$ is the minimum $k$ such that every induced subgraph of $G$ contains a vertex of degree at most $k$.

%\footnote{The {\it degeneracy} of a graph $G$
%is the minimum $k$ such that every induced subgraph of $G$ contains a vertex of degree at most $k$. In particular, 
%degeneracy is bounded in all proper minor-closed classes of graphs.}

Spinrad showed in \cite{Spinrad} that bounded clique-width also yields an implicit representation. The recently introduced 
parameter {\it twin-width} generalizes clique-width in the sense that bounded clique-width implies bounded twin-width,
but not vice versa. It was shown in \cite{twin} that bounded twin-width also implies the existence of an implicit representation. 

The notion of graph functionality, introduced in \cite{functionality},  generalizes both degeneracy and twin-width 
in the sense that bounded degeneracy or bounded twin-width implies bounded functionality, but not vice versa. 
The graphs of bounded functionality have at most factorial speed of growth \cite{implicit}. However,
whether they admit an implicit representation is wide-open. To approach this question, in Section~\ref{sec:sd} 
we analyse a parameter intermediate between twin-width and functionality. Then in Sections~\ref{sec:cpn} and~\ref{sec:dspn}, we introduce more parameters 
and report some results concerning them.  

%%%%%%%%%%%%%%%%%%%%%%%%%%%%%%%%%%%
\subsection{Symmetric difference}
\label{sec:sd}
%%%%%%%%%%%%%%%%%%%%%%%%%%%%%%%%%%%

Let $G$ be a graph. Given two vertices $x, y$, we define the {\em symmetric difference} of $x$ and $y$ in $G$
as the number of vertices in $V(G)- \{x, y\}$ adjacent to exactly one of $x$ and $y$, and we denote 
it by $\sd(x, y)$. We define the symmetric difference $\sd(G)$ of $G$ as the smallest number such that 
any induced subgraph of $G$ has a pair of vertices with symmetric difference at most $\sd(G)$. 

This parameter was introduced in \cite{functionality}, where it was shown that bounded clique-width implies bounded symmetric difference.
Paper \cite{functionality} also identifies a number of classes of bounded symmetric difference. Below we reveal more classes 
where this parameter is bounded.

The first result deals with classes of graphs of bounded {\em contiguity}. This includes, for instance, bipartite permutation graphs, 
which have contiguity 1 \cite{branstadt-graph-classes}. The notion of contiguity was introduced in \cite{contiguity} and was motivated 
by the need for compact representations of graphs in computer memory. 
One approach to achieving this goal is finding a linear order of the vertices in which the neighbourhood of each vertex forms an interval. 
Not every graph admits such an ordering, in which case one can relax this requirement by looking for an ordering in which the neighbourhood 
of each vertex can be split into at most $k$ intervals. The minimum value of $k$ which allows a graph $G$ to be represented in this way is the {\it contiguity} of $G$, denoted $\cont(G)$.  

\begin{theorem}\label{thm:symdif}
	For any $k \geq 1$, any graph of contiguity $k$ has symmetric difference at most $2k$.
\end{theorem}

\begin{proof}
	It suffices to show that any graph $G$ of contiguity $k$ has a pair of vertices with symmetric difference at most $2k$. 
Let $x_1, \dots, x_n$ be a linear order of the vertices in which the neighbourhood of every vertex consists of at most $k$ intervals, and let
	$$
		S := \sum\limits_{i = 1}^{n - 1} \sd(x_i, x_{i + 1}).
	$$
	Since the neighbourhood of an arbitrary vertex $y$ consists of at most $k$ intervals in the linear order, there are 
	at most $2k$ pairs of consecutive vertices $x_i, x_{i+1}$ such that $y$ is adjacent to one of them, but not
	adjacent to the other. Therefore, $y$ contributes at most $2k$ to $S$, and hence $S \leq 2kn$.
	Since there are $n - 1$ terms in the sum, one of them must be at most $2k$ (unless $n \leq 2k + 1$, in which case the statement is trivial).
\end{proof}

The second result deals with classes of $F_{t,p}$-free bipartite graphs (see Figure~\ref{fig:speed} for an illustration of $F_{t,p}$). 
These classes have unbounded clique-width for all $t,p\ge 2$. To show that they have bounded symmetric difference,
we assume without loss of generality that $t=p$.

\begin{theorem}\label{thm:ftt}
For each $t\ge 2$, every $F_{t,t}$-free bipartite graph $G=(B,W,E)$ has symmetric difference at most $2t$.
\end{theorem}

\begin{proof}
It is sufficient to show that $G$ has a pair of vertices with symmetric difference at most $2t$.
For two vertices $x,y$, we denote by $\dd(x,y)$ the degree difference $|\deg(x)-\deg(y)|$ and for a subset $U\subseteq V(G)$, we write $\dd(U):=\max\{\dd(x,y)\ :\  x,y\in U\}$. 
Assume without loss of generality that $\dd(W)\le \dd(B)$ and let $x,y$ be two vertices in $B$ with $\dd(x,y)=\dd(B)$, $\deg(x)\ge \deg(y)$. 

Write $X:=N(x)-N(y)$. Clearly, $\dd(B)\le |X|$. If $|X|\le 2$, then $\sd(x,y)\le 4\le 2t$ and we are done.

%Now assume $|X|\ge 3$. Since $\dd(X)\le \dd(W)\le \dd(B)\le |X|$, the set $X$ contains two vertices $p$ and $q$ with $\dd(p,q)\le 1$. 
%Then $\sd(p,q)\le 2t$, since otherwise both $P:=N(p)-N(q)$ and $Q:=N(q)-N(p)$ have size at least $t$, in which case $x,y,p,q$ together 
%with $t$ vertices from $P$ and $t$ vertices from $Q$ induce the forbidden graph $F_{t,t}$.  

Now assume $|X|\ge 3$. First observe that $\dd(X)\le \dd(W)\le \dd(B)\le |X|$. Let $x_1, x_2, \ldots, x_{|X|}$ be a sequence of
vertices of $X$ with decreasing degree order, i.e. $\deg(x_i) \geq \deg(x_j)$ for all $i>j$. 
Since $\sum_{i=1}^{|X|-1} (\deg(x_i) - \deg(x_{i+1})) = \dd(X) \leq |X|$, and each summand is non-negative, by the Pigeonhole principle 
it easily follows that $d(x_i)-d(x_{i+1}) \le 1$ for some $i$. So $X$ contains two vertices 
$p$ and $q$ with with $\dd(p,q)\le 1$. Finally, we conclude that $\sd(p,q)\le 2t$, since otherwise both $P:=N(p)-N(q)$ 
and $Q:=N(q)-N(p)$ have size at least $t$, in which case $x,y,p,q$ together 
with $t$ vertices from $P$ and $t$ vertices from $Q$ induce the forbidden graph $F_{t,t}$.

\end{proof}

The symmetric difference is also bounded in the class of $S_{1,2,3}$-free graphs, since these graphs have bounded clique-width \cite{skew-star}.
For the remaining class from Theorem~\ref{thm:speed-dichotomy}, i.e.\ the class of $P_7$-free bipartite graphs, the boundedness of symmetric difference 
is an open question.

\begin{conjecture}\label{con:1}
The symmetric difference is bounded in the class of $P_7$-free bipartite graphs.
\end{conjecture}

We also conjecture that every class of graphs of bounded symmetric difference admits an implicit representation.

\begin{conjecture}\label{con:2}
Every class of graphs of bounded symmetric difference admits an implicit representation. 
\end{conjecture}

We will verify this conjecture for the classes of $F_{t,p}$-free bipartite graphs
in Section~\ref{sec:implicit}. 

%%%%%%%%%%%%%%%%%%%%%%%%%%%%%%%%%%%%%%%%%%%%%%%%%%%%%%%%%%%
\subsection{Chain partition number}
\label{sec:cpn}
%%%%%%%%%%%%%%%%%%%%%%%%%%%%%%%%%%%%%%%%%%%%%%%%%%%%%%%%%%%

%A \emph{chain graph} is a bipartite graph whose vertices in one of the colour classes can be linearly ordered with respect to the inclusion of their neighbourhoods. The class of chain graphs is precisely the class of $2K_2$-free bipartite graphs. 

Let $G$ be a graph and let $k$ be the minimum number of subsets in a partition of $V(G)$ into homogeneous sets 
such that the edges between any pair of subsets form a chain graph. We call $k$ the {\it chain partition number} of $G$.  

This notion was never formally introduced in the literature, but it implicitly appeared in \cite{Aistis}, where the author 
proved a result which can be stated as follows. 

\begin{theorem}{\rm \cite{Aistis}}\label{thm:cpn}
The chain partition number is unbounded in a hereditary class $\mathcal X$ if and only if $\mathcal X$ contains at least one of the following six classes:
the class $\mathcal M$ of (bipartite) graphs of vertex degree at most 1, the class $\widetilde{\mathcal M}$ of the bipartite complements of graphs in $\mathcal M$, 
the classes of complements of graphs in $\mathcal M$ and $\widetilde{\mathcal M}$, 
and two related subclasses of split graphs obtained from graphs in $\mathcal M$ and $\widetilde{\mathcal M}$ by creating a clique in one of the parts of their bipartition.
\end{theorem}

Bounded chain partition number implies implicit representation by Theorem~\ref{thm:local} and the fact that chain graphs admit an implicit representation. 

%%%%%%%%%%%%%%%%%%%%%%%%%%%%%%%%%%%%%%%%%%%%%%%%%%%%%%%%%%%
\subsection{Double-star partition number}
\label{sec:dspn}
%%%%%%%%%%%%%%%%%%%%%%%%%%%%%%%%%%%%%%%%%%%%%%%%%%%%%%%%%%%

The results in \cite{Aistis} suggest one more parameter that generalizes the chain partition number.
To define this parameter, let us call a class $\mathcal X$ of bipartite graphs {\it double-star-free} if there is a constant 
$p$ such that no graph $G$ in $\mathcal X$ contains an unbalanced copy of $2K_{1,p}$, i.e.\ an induced copy of $2K_{1,p}$ 
in which the centres of both stars belong to the same part of the bipartition of $G$. In particular, 
every class of {\it double-star-free} graphs is $F_{t,p}$-free for some $t,p$. We will say that a class $\mathcal X$ of 
graphs is of bounded {\it double-star partition number} if there are constants $k$ and $p$ such that 
the vertices of every graph in $\mathcal X$ can be partitioned into at most $k$ homogeneous subsets
such that the edges between any pair of subsets form a bipartite graph that does not contain an unbalanced copy of $2K_{1,p}$.
We observe that if $p=1$, we obtain a class of bounded chain partition number.

Classes of bounded double-star partition number have been defined in the previous paragraph through two constants, $k$ and $p$. By taking the maximum of the two, 
we can talk about a single constant, which can be viewed as a graph parameter defining the family of classes of bounded double-star partition number.
In Section~\ref{sec:implicit} we will show that any class in this family admits an implicit representation. 

\medskip
Similarly to Theorem~\ref{thm:cpn}, classes of bounded double-star partition number admit a characterisation in terms of minimal hereditary classes 
where the parameter is unbounded. In this characterisation, the class $\mathcal M$ of graphs of vertex degree at most 1 is replaced by 
the class $\mathcal S$ of star forests in which the centers of all stars belong to the same part of the bipartition.
% \begin{itemize}
% \item[$\mathcal S$] of star forests in which the centers of all stars belong to the same part of the bipartition. 
% \end{itemize}

Seven more classes are obtained 
by various complementations either between the two parts of the bipartition or within these parts. Together with the class $\mathcal S$ itself this gives 
eight different classes of bipartite, co-bipartite and split graphs  (notice that complementing the part containing 
the centers of the stars and the part containing the leaves of the stars produce different classes of graphs). We will refer to all of them as the
``classes related to $\mathcal S$''.

\begin{theorem}{\rm \cite{Aistis}}\label{thm:dspn}
A hereditary class $\mathcal X$ is of unbounded double-star partition number if and only if $\mathcal X$ contains at least one of the following ten classes:
the class of $P_3$-free graphs (disjoint union of cliques), the class of $\overline{P}_3$-free graphs (complete multipartite graphs),   
the class $\mathcal S$ and the seven classes related to $\mathcal S$.
\end{theorem}

%%%%%%%%%%%%%%%%%%%%%%%%%%%%%%%%%%%%%%%%%%%%%%%%%%%%%%%%%%%
\subsection{$h$-index}
\label{sec:h-index}
%%%%%%%%%%%%%%%%%%%%%%%%%%%%%%%%%%%%%%%%%%%%%%%%%%%%%%%%%%%

We observed earlier that the family of classes of bounded chain partition number forms a subfamily of classes of bounded double-star partition number.
One other interesting restriction of the latter family consists of classes of bounded $h$-index. 
The $h$-index $h(G)$ of a graph $G$ is the largest $k \geq 0$ such that $G$ has $k$ vertices of degree at least $k$. This parameter is important in the study of 
dynamic algorithms \cite{hindex}. 

To see that bounded $h$-index does indeed imply bounded double-star partition number, we note that all classes in 
Theorem~\ref{thm:dspn}, except star forest, contain either all complete graphs or all complete bipartite graphs. 
The implication then follows since complete graphs, complete bipartite graphs and star forests have unbounded $h$-index (which gives us the contrapositive statement).

In fact, as we show in Theorem~\ref{thm-h-index} below, those three classes are the only minimal classes of unbounded $h$-index. 
We note that a ``prototype'' of this characterisation appeared in \cite{cographs}, where the set of minimal classes of unbounded $h$-index was identified {\it within the class of cographs}.
It was also conjectured in \cite{cographs} that this characterisation extends to the universe of all graphs, which is what we now show.

\begin{theorem}\label{thm-h-index}
	The classes of star forests, complete bipartite graphs and complete graphs are the only three minimal hereditary classes of graphs of unbounded $h$-index.
\end{theorem}

\begin{proof}
	It is routine to check that $h$-index is unbounded in these three classes. Hence it remains to show that any class $X$ for which $h$-index is unbounded contains 
	one of the three classes. To do this, it suffices to show that for each $n$, there exists $d = d(n)$ such that any graph of $h$-index $d$ or greater contains, 
as an induced subgraph, either a clique $K_n$, or a complete bipartite graph $K_{n, n}$, or a star forest $nK_{1, n}$. We start with the following observation:
	\begin{itemize}
		\item[(*)] For any pair of positive integers $n,m$, there exists an integer $z_{n,m}$ such that any bipartite graph
		with $z_{n,m}$ vertices of degree at least $z_{n,m}$ in one side of its bipartition contains either a
		$K_{n,n}$, or an induced star forest $mK_{1,n}$.  
	\end{itemize}
	
	We prove the observation by induction on $m$. The statement is true for $m=1$, as we can simply take $z_{n,1}:=n$ for any $n$. 
Now let $m>1$. Write $R^b(i, j)$ for the bipartite Ramsey number, i.e. the smallest integer such that any bipartite graph with $R^b(i, j)$ 
vertices in each side either contains $K_{i, i}$, or the bipartite complement of $K_{j, j}$. 
Put $z_{n, m} := R^b(n, z_{n, m-1})$, and let $G = (A, B, E)$ be a bipartite graph containing a set $S \subseteq A$ of $z_{n,m}$ vertices of degree at least $z_{n,m}$.
	 
	Pick a vertex $s \in S$ of minimum degree, and consider the bipartite graph $G'$ induced by the sets $S$ and $N(s)$. 
	As both $|S| \geq z_{n,m}$ and $|N(s)| \geq z_{n,m}$, $G'$ contains by construction either a $K_{n,n}$, or 
	two sets $S' \subseteq S$, $T \subseteq N(s)$ of size $z_{n,m-1}$ with no edges between them. In the former case we are done, so
	consider the latter. Note that, by minimality of the degree of $s$, since each vertex in $S'$ has at least $|T|=z_{n,m-1}$ 
non-neighbours in $N(s)$, each vertex in $S'$ must have at least $z_{n,m-1}$ neighbours outside $N(s)$. 
Let $G''$ be the graph induced by the vertices in $S'$, together with their neighbourhoods outside of $N(s)$. 
Applying the induction hypothesis to $G''$, we find that $G''$ contains either a $K_{n,n}$, or an induced $(m-1)K_{1,n}$. 
In the former case we are once more done; in the latter, we note that adjoining vertex $s$ together with any $n$ vertices from $T$ to the $(m-1)K_{1, n}$
	 yields a $mK_{1,n}$, and the observation is proven.
	
	\bigskip
	
	The second ingredient is as follows:
	\begin{itemize}
		\item[(**)] For any positive integer $n$, there exists an integer $m=m(n)$ with the following property: 
if $G$ is a graph whose vertex set can be partitioned into $m$ independent sets of size $n$, then $G$ contains a $K_n$, 
an induced $K_{n,n}$, or an independent set of size $n^2$ which is the union of $n$ of the original independent sets.  
	\end{itemize}
	
	To show this, write $R_c(i)$ for the multicolour Ramsey number -- the smallest integer such that, 
for any edge colouring with $c$ colours of a complete graph on $R_c(i)$ vertices, there is a monochromatic clique on $i$ vertices. 
%has its edges coloured with  $m=R_{2^{n^2}}(2n)$, be a smallest integer $m$ such that whenever a complete graph on $m$ vertices have edges coloured using $2^{n^2}$ colours,
	%one can always find a monochromatic clique of size $2n$. 
	Put $m:=R_{2^{n^2}}(2n)$.
	
	Now let $G$ be a graph with vertex set $V(G)=V_1 \cup V_2 \cup \ldots \cup V_m$ (the $V_i$ are disjoint), such that for all $i$, $V_i$ is independent and $|V_i|=n$.
	For each $i$, fix an ordering of the vertices in $V_i$, that is, a bijection $\varphi_i : V_i \to [n]$. 
For each $i, j$ with $1 \leq i < j \leq m$, put $E_{ij} := \{(\varphi_i(x), \varphi_j(y)) : x \in V_i, y \in V_j, \text{ and } \{x, y\} \in E(G)\}$. 
Intuitively, $E_{ij} \subseteq [n] \times [n]$ is simply the edge set between $V_i$ and $V_j$, where we orient the edges from the lower to the higher index, 
and identify the two sets with copies of $[n]$ via their respective orderings. 
	
	Consider an auxiliary complete graph with vertex set $[m]$; for each $i < j$, we colour the edge $\{i, j\}$ with the set $E_{ij}$. 
We note that there are $2^{n^2}$ possible colours, corresponding to the subsets of $[n] \times [n]$; we find a monochromatic clique on $2n$ vertices $i_1 < i_2 < \dots < i_{2n}$. 
Note that $E_{i_1 i_2}=E_{i_k i_l}$ for all $1 \leq k < l \leq 2n$. 
If $E_{i_1 i_2}$ is empty, then $V_{i_1} \cup V_{i_2} \cup \dots \cup V_{i_n}$ induces an independent set of size $n^2$. 
If there exists $1 \leq t \leq n$ such that $E_{i_1 i_2}$ contains $(t, t)$, then $\{\varphi^{-1}_{i_r}(t) : 1 \leq r \leq n\}$ is a $K_n$. 
Finally, if for some $s \neq t$,  $E_{i_1 i_2}$ contains $(s, t)$ and not $(s, s)$ nor $(t, t)$,
then $\{\varphi^{-1}_{i_r}(s) : 1 \leq r \leq n\} \cup \{\varphi^{-1}_{i_r}(t) : n + 1 \leq r \leq 2n\}$ induces a $K_{n, n}$ in $G$.

	\bigskip
	
	We can now put these two facts together to obtain our main result. 
	
	\begin{itemize}
		\item[(***)] For any integer $n$, there exists an integer $d=d(n)$ such that any graph of $h$-index at least $d$
		contains either a $K_n$, a $K_{n,n}$, or a $nK_{1,n}$. 
	\end{itemize}
	
	Write $R(p, q)$ for the usual Ramsey number -- the smallest number such that a graph on $R(p, q)$ vertices contains either a clique of size $p$, 
or an independent set of size $q$. Put $N:=R(n, n)$, and let $d:=R(n, z_{N, m(n)})$, where $z$ and $m$ are defined as in (*) and (**) respectively. 
	Suppose that a graph $G$ contains a set $S$ of $d$ vertices of degree at least $d$. Then either $G[S]$ contains a $K_n$, in which case we are done, 
	or it contains an independent set $S' \subseteq S$ of size $z=z_{N, m(n)}$. Note that each vertex of $S'$ has degree at least 
	$z_{N,m(n)}$ outside $S$. Write $G'$ for the bipartite graph with parts $S'$ and $T' := V(G) \setminus S'$, 
and with $E(G') := \{\{x, y\} \in E(G) : x \in S', y \in T'\}$ (that is, $G$ is obtained from $G'$ by removing all edges with both endpoints in $T'$). 
By definition of $z$, $G'$ contains either a $K_{N, N}$, or an induced star forest
	$m(n)K_{1,N}$. A $K_{N,N}$ in this bipartite graph translates into either an induced $K_{n,n}$ in $G$ or a $K_n$ in $G$; 
indeed, look at the $N$ vertices of the $K_{N, N}$ lying in $T'$: by construction, the graph induced by them in $G$ must contain an independent set of size $n$, or a clique of size $n$.
	
	It remains to consider the case where $G'$ contains an induced star forest $m(n)K_{1, N}$. 
Note that each $K_{1, N}$ is induced in $G'$, but not necessarily in $G$. 
However, by a similar argument to the case above, either $G$ contains a $K_n$ (and we are done), 
or we may find a $m(n)K_{1, n}$ induced in $G'$ and not necessarily in $G$, 
but such that each separate $K_{1, n}$ is also induced in $G$. 
In particular, the leaves of each separate $K_{1, n}$ form an independent set in $G$. 
Consider the graph $H \subseteq G$ induced by the leaves of those $m(n)$ $K_{1, n}$'s. 
By definition of $m(n)$, $H$ either contains a $K_n$ (and we are done), 
an induced $K_{n, n}$ (and we are done), or the $n^2$ leaves of $n$ of the $K_{1, n}$'s 
induce an independent set. In this final case, $G$ contains an induced $nK_{1, n}$, and this proves the theorem.
	
\end{proof}

%%%%%%%%%%%%%%%%%%%%%%%%%%%%%%%%%%%%%%%%%%%%%%%%%%%%%%%%%%%%%%%%%%%%%%%%%%%%%%%%%%%%%%%%%%%%%%%%%%%%%%%%%%%%%%%%%%%%%%%%%%%%%%%%%%%%%%%%%%%%%%%%%%%%%%%%%%%%%%%%%%%%%%%%%%%%
%%%%%%%%%%%%%%%%%%%%%%%%%%%%%%%%%%%%%%%%%%%%%%%%%%%%%%%%%%%%%%%%%%%%%%%%%%%%%%%%%%%%%%%%%%%%%%%%%%%%%%%%%%%%%%%%%%%%%%%%%%%%%%%%%%%%%%%%%%%%%%%%%%%%%%%%%%%%%%%%%%%%%%%%%%%%

\section{Implicit representations}
\label{sec:implicit}

%%%%%%%%%%%%%%%%%%%%%%%%%%%%%%%%%%%%%%%%%%%%%%%%%%%%%%%%%%%%%%%%%%%%%%%%%%%%%%%%%%%%%%%%%%%%%%%%%%%%%%%%%%%%%%%%%%%%%%%%%%%%%%%%%%%%%%%%%%%%%%%%%%%%%%%%%%%%%%%%%%%%%%%%%%%%
%%%%%%%%%%%%%%%%%%%%%%%%%%%%%%%%%%%%%%%%%%%%%%%%%%%%%%%%%%%%%%%%%%%%%%%%%%%%%%%%%%%%%%%%%%%%%%%%%%%%%%%%%%%%%%%%%%%%%%%%%%%%%%%%%%%%%%%%%%%%%%%%%%%%%%%%%%%%%%%%%%%%%%%%%%%%
In this section, we identify a number of new hereditary classes of graphs that admit an implicit representation.

\subsection{$F_{t,p}$-free bipartite graphs}\label{sec:Ftp}

In this section we show that $F_{t,p}$-free bipartite graphs admit an implicit representation for any $t$ and $p$.
Together with Theorem~\ref{thm:ftt} this verifies Conjecture~\ref{con:2} for these classes. 

Without loss of generality we assume that $t=p$ and split the analysis into several intermediate steps.
The first step deals with the case of double-star-free bipartite graphs.

\begin{lemma}
Let $G = (A, B, E)$ be a bipartite graph that does not contain an unbalanced induced copy of $2K_{1, t}$. Then $G$ has a vertex of degree at most $t - 1$ or bi-codegree at most $(t - 1)(t^2 - 4t + 5)$.
\end{lemma}

\begin{proof}
Let $x \in A$ be a vertex of maximum degree. Write $Y$ for the set of neighbours of $x$, and $Z$ for its set of non-neighbours in $B$ (so $B = Y \cup Z$). 
We may assume $|Y| \geq t$ and $|Z| \geq (t - 1)(t^2 - 4t + 5) + 1$, since otherwise we are done. 
		
Note that any vertex $w \in A$ is adjacent to fewer than $t$ vertices in $Z$. Indeed, if $w \in A$ has $t$ neighbours in $Z$, 
then it must be adjacent to all but at most $t - 1$ vertices in $Y$ (since otherwise a $2K_{1, t}$ appears), so its degree is greater than that of $x$, a contradiction.
		
We now show that $Z$ has a vertex of degree at most $t - 1$. Pick members $z_1, \dots, z_{t - 1} \in Z$ in a non-increasing order of their degrees, 
and write $W_i$ for the neighbourhood of $z_i$. Since $G$ does not contain an unbalanced induced copy of $2K_{1, t}$ and $\deg(z_{i+1})\le \deg(z_i)$, for all $1 \leq i \leq t - 2$, $|W_{i + 1} - W_i| \leq t - 1$. 
It is not difficult to see that in fact, $|W_{i + 1} - \bigcap\limits_{s = 1}^i W_s| \leq (t - 1)i$, 
and in particular, $|W_{t - 1} - \bigcap\limits_{i = 1}^{t - 2} W_i| \leq (t - 1)(t - 2)$. 
		
With this, we can compute an upper bound on the number of vertices in $Z$ which have neighbours in $W_{t - 1}$: 
by the degree condition given above, each vertex in $W_{t - 1} \cap \bigcap\limits_{i = 1}^{t - 2} W_i$ 
is adjacent to no vertices in $Z$ other than $z_1, \dots, z_{t - 1}$. 
Each of the at most $(t - 1)(t - 2)$ vertices in $W_{t - 1} - \bigcap\limits_{i = 1}^{t - 2} W_i$ 
has at most $t - 2$ neighbours in $Z$ other than $z_{t - 1}$. This accounts for a total of 
at most $(t - 1) + (t - 1)(t - 2)^2 = (t - 1)(t^2 - 4t + 5)$ vertices which have neighbours in $W_{t - 1}$, 
including $z_{t - 1}$ itself. By assumption on the size of $Z$, there must be a vertex $z \in Z$ which has no common neighbours with $z_{t - 1}$. 
Since $2K_{1, t}$ is forbidden, one of $z$ and $z_{t - 1}$ has degree at most $t - 1$, as claimed.
\end{proof}

An immediate implication of this result, combined with Theorem~\ref{thm:partial-bip} (applied with $B_2=\emptyset$), is that double-star-free bipartite graphs admit an implicit representation.
%The only remark that needs to be done is that Theorem~\ref{thm:partial} can be easily adapted to the bipartite case by replacing vertices of bounded co-degree with vertices of bounded bi-codegree.  	

\begin{corollary}\label{cor:2stars}
The class of bipartite graphs excluding an unbalanced induced copy of $2K_{1, t}$ admits an implicit representation for any fixed $t$.
\end{corollary}

Together with Theorem~\ref{thm:local} this corollary implies one more interesting conclusion.	
	
\begin{corollary}
The classes of graphs of bounded double-star partition number, and in particular those of bounded $h$-index, admit an implicit representation.
\end{corollary}

%In the context of bipartite graphs, this corollary together with Theorem~\ref{thm:dspn} implies the following generalization of Corollary~\ref{cor:2stars}.

%\begin{corollary}
%Every class of bipartite graphs excluding a star forest 
%and the bipartite complement of a star forest admits an implicit representation.
%\end{corollary}

\medskip
Our next step towards implicit representations of $F_{t,t}$-free bipartite graphs deals with the case of $F^1_{t,t}$-free bipartite graphs,
where $F^1_{t,t}$ is the graph obtained from $F_{t,t}$ by deleting the isolated vertex.

\begin{lemma}\label{lem:ftt-1}
The class of $F^1_{t,t}$-free bipartite graphs admits an implicit representation.
\end{lemma}

\begin{proof}
It suffices to prove the result for connected graphs (this follows for instance from Theorem~\ref{thm:local}).
Let $G$ be a connected $F^1_{t,t}$-free bipartite graph and let $v$ be a vertex of maximum degree in $G$. We denote by
$V_i$ the set of vertices at distance $i$ from $v$.

First, we show that the subgraph $G[\{v\}\cup V_1\cup V_{2}]$ admits an implicit representation. To this end, we denote by $u$
a vertex of maximum degree in $V_1$, by $U$ the neighbourhood of $u$ in $V_2$, $W:=V_2-U$, and $V'_1:=V_1-\{u\}$. 

Let $x$ be a vertex in $V'_1$ and assume it has $t$ neighbours in $W$. Then
$x$ has at least $t$ non-neighbours in $U$ (due to maximality of $u$), in which case the $t$ neighbours of $x$ in $W$, the $t$ non-neighbours of $x$ in $U$ together with
$x$, $u$  and $v$ induce an $F^1_{t,t}$. This contradiction shows that every vertex of $V'_1$ has at most $t-1$ neighbours in $W$, 
and hence the graph $G[V'_1\cup W]$ admits an implicit representation by Theorem~\ref{thm:partial-bip} (applied with $B_2=\emptyset$).

To prove that $G[V'_1\cup U]$ admits an implicit representation, we observe that this graph does not contain an unbalanced  induced copy of $2K_{1,t}$. Indeed, if the centers of the two stars belong to $V'_1$, 
then they induce an $F^1_{t,t}$ together with vertex $v$, and if the centers of the two stars belong to $U$, then they induce an $F^1_{t,t}$ together with vertex $u$.
Therefore, the graph $G[\{v\}\cup V_1\cup V_{2}]$ can be covered by at most four graphs (two of them being the stars centered at $v$ and $u$), each of which admits an implicit representation,
and hence by Theorem~\ref{thm:local} this graph admits an implicit representation. 

To complete the proof, we observe that every vertex of $V_2$ has at most $t-1$ neighbours in $V_3$ (note that, by the definition of the sets $V_i$,
any neighbour of $V_2$ is either in $V_1$ or in $V_3$). Indeed, if a vertex $x\in V_2$ has $t$ neighbours in $V_3$, then
$x$ has at least $t$ non-neighbours in $V_1$ (due to maximality of $v$), in which case the $t$ neighbours of $x$ in $V_3$, the $t$ non-neighbours of $x$ in $V_1$ together with
$x$, $v$, and any neighbour of $x$ in $V_1$ (which must exist by definition) induce an $F^1_{t,t}$. 

Now we apply Theorem~\ref{thm:partial} with $A=\{v\}\cup V_1\cup V_2$ to conclude that $G$ admits an implicit representation, because every vertex of $A$ has at most $t-1$ neighbours outside of $A$.    
\end{proof}

The last step towards implicit representations of $F_{t,t}$-free bipartite graphs is similar to Lemma~\ref{lem:ftt-1} with some modifications.

\begin{theorem}\label{thm:ftt-0}
The class of $F_{t,t}$-free bipartite graphs admits an implicit representation.
\end{theorem}

\begin{proof}
By analogy with Lemma~\ref{lem:ftt-1} we consider a {\it connected} $F_{t,t}$-free bipartite graph $G$, denote by $v$  a vertex of maximum degree in $G$ and by
$V_i$ the set of vertices at distance $i$ from $v$. Also, we denote by $u$
a vertex of maximum degree in $V_1$, by $U$ the neighbourhood of $u$ in $V_2$, $W:=V_2-U$, and $V'_1:=V_1-\{u\}$. 

Let $x$ be a vertex in $V'_1$ and assume it has $t$ neighbours and one non-neighbour $y$ in $W$. Then
$x$ has at least $t$ non-neighbours in $U$ (due to maximality of $u$), in which case the $t$ neighbours of $x$ in $W$, the $t$ non-neighbours of $x$ in $U$ together with
$x$, $y$, $u$  and $v$ induce an $F_{t,t}$. This contradiction shows that every vertex of $V'_1$ has either at most $t-1$ neighbours or at most $0$ non-neighbours in $W$, 
and hence the graph $G[V'_1\cup W]$ admits an implicit representation by Theorem~\ref{thm:partial-bip} (applied with $B_2=\emptyset$).

To prove that $G[V'_1\cup U]$ admits an implicit representation, we show that this graph is $\widetilde{F}^1_{t,t}$-free
(we emphasize that in $\widetilde{F}^1_{t,t}$ the isolated vertex belongs to one part of the bipartition and the centers of the stars to the other part). 
Indeed, if the centers of the two stars of $\widetilde{F}^1_{t,t}$ belong to $V'_1$, 
then $\widetilde{F}^1_{t,t}$ together with vertex $v$ induce an $F_{t,t}$, and 
if the centers of the two stars of $\widetilde{F}^1_{t,t}$ belong to $U$, 
then $\widetilde{F}^1_{t,t}$ together with vertex $u$ induce an $F_{t,t}$. 
Therefore, the graph $G[\{v\}\cup V_1\cup V_{2}]$ can be covered by at most four graphs (two of them being the stars centered at $v$ and $u$), each of which admits an implicit representation,
and hence by Theorem~\ref{thm:local} this graph admits an implicit representation. 

To complete the proof, we observe that every vertex of $V_2$ has either at most $t-1$ neighbours or $0$ non-neighbours in $V_3$. Indeed, if a vertex $x\in V_2$ has $t$ neighbours 
and one non-neighbour $y$ in $V_3$, then
$x$ has at least $t$ non-neighbours in $V_1$ (due to maximality of $v$), in which case the $t$ neighbours of $x$ in $V_3$, the $t$ non-neighbours of $x$ in $V_1$ together with
$x$, $y$, $v$, and any neighbour of $x$ in $V_1$ induce an $F_{t,t}$. 

Finally, we observe that if  a vertex $x\in V_2$ has $t$ neighbours in $V_3$, then $V_5$ (and hence $V_i$ for any $i\ge 5$) is empty, because otherwise an induced $F_{t,t}$
arises similarly as in the previous paragraph, where vertex $y$ can be taken from $V_5$. 
Now we apply Theorem~\ref{thm:partial-bip} (with $B_2=\emptyset$) with $A=\{v\}\cup V_1\cup V_2$ to conclude that $G$ admits an implicit representation.
Indeed, if each vertex of $V_2$ has at most $t-1$ neighbours in $V_3$, then each vertex of $A$ has at most $t-1$ neighbours outside of $A$,
and if a vertex of $V_2$ has at least $t$ neighbours in $V_3$, then $V_i=\emptyset$ for $i\ge 5$ and hence every vertex of $A$ has at most $t-1$ neighbours or at most $0$ non-neighbours in the {\it opposite}
part outside of $A$.     
\end{proof}

%%%%%%%%%%%%%%%%%%%%%%%%%%%%%%%%%%%%%%%%%%%%%%%%%%%%%%%%%%%%
\subsection{One-sided forbidden induced bipartite subgraphs}
%%%%%%%%%%%%%%%%%%%%%%%%%%%%%%%%%%%%%%%%%%%%%%%%%%%%%%%%%%%%

In the context of bipartite graphs, some hereditary classes are defined by forbidding one-sided copies of bipartite graphs. 
Consider, for instance, the class of star forests, whose vertices are partitioned into an independent set of black vertices and an independent set of white vertices.
If the centers of all stars have the same colour, say black, then this class is defined by 
forbidding a $P_3$ with a white center. Very little is known about implicit representations for classes defined by one-sided forbidden induced bipartite subgraphs.
It is known, for instance, that bipartite graphs without a one-sided $P_5$ admit an implicit representation. This is not difficult to show and also follows 
from the fact $P_6$-free bipartite graphs have bounded clique-width and hence admit an implicit representation (note that $P_6$ is symmetric with respect to swapping the bipartition). 
Below we strengthen the result for one-sided forbidden $P_5$ to one-sided forbidden $F_{t,1}$.
We start with a one-sided forbidden $F^1_{t,1}$, where again $F^1_{t,1}$ is the graph obtained from $F_{t,1}$ by deleting the isolated vertex.

\begin{lemma}\label{lem:F1t1}
The class of bipartite graphs containing no one-sided copy of $F^1_{t,1}$ admits an implicit representation. 
\end{lemma}

\begin{proof}
Let $G=(U,V,E)$ be a bipartite graph containing no copy of $F^1_{t,1}$ with the vertex of largest degree in $U$.
To prove the lemma, we apply Theorem~\ref{thm:partial-bip}. 

If $G$ is edgeless, then the conclusion trivially follows from Theorem~\ref{thm:partial-bip} with $A=V(G)$.  
Otherwise, let $u$ be a vertex of maximum degree in $U$. We split the vertices of $V$ into the set $V_1$ of neighbours and the set $V_0$ of non-neighbours of $u$.
Assume there exists a vertex $x\in U$ that has neighbours both in $V_1$ and in $V_0$. We denote by  $V_{10}$ the set of non-neighbours of $x$ in $V_1$
and by $V_{01}$ the set of neighbours of $x$ in $V_0$.
We note that $|V_{01}|\le |V_{10}|$, since $\deg(x)\le \deg(u)$. Additionally, $|V_{10}|<t$, since otherwise $t$ vertices in $V_{10}$, a vertex in $V_{01}$ and a common neighbour of $u$ and $x$ 
(these vertices exist by assumption) together with $u$ and $x$ induce a forbidden copy of $F^1_{t,1}$. Therefore, $x$ has at most $t-1$ non-neighbours in $V_1$ and at most $t-1$ neighbours in $V_0$. 
Now we define three subsets $A,B_1,B_2$ as follows:
\begin{itemize}
\item[] $A$ consists of vertex $u$, the vertices of $U$ that have neighbours both in $V_1$ and in $V_0$, and the vertices of $U$ that have neighbours neither in $V_1$ nor in $V_0$,
\item[] $B_1$ consists of $V_1$ and the vertices of $U$ that have neighbours only in $V_1$,
\item[] $B_2$ consists of $V_0$ and the vertices of $U$ that have neighbours only in $V_0$.
\end{itemize}
With this notation, the result follows from Theorem~\ref{thm:partial-bip}. 
\end{proof}

\begin{theorem}\label{thm:F1t1}
The class of bipartite graphs containing no one-sided copy of $F_{t,1}$ admits an implicit representation. 
\end{theorem}

\begin{proof}
Let $G=(U,V,E)$ be a connected bipartite graph containing no one-sided copy of $F_{t,1}$ with the vertex of largest degree in $U$. 
Let $v$ be a vertex in $V$ and let $V_i$ the set of vertices at distance $i$ from $v$. 
Then the graph $G_1:=G[V_1\cup V_2]$ does not contain a one-sided copy of $\widetilde{F}^1_{t,1}$ with the vertex of largest degree in $V_1$
(we emphasize that in $\widetilde{F}^1_{t,1}$ the isolated vertex belongs to one part of the bipartition
and the vertex of degree $t$  to the other part). Indeed, a one-sided copy of $\widetilde{F}^1_{t,1}$ with the vertex of largest degree in $V_1$
together with $v$ would induce a one-sided copy of $F_{t,1}$ with the vertex of largest degree in $U$. 
Therefore, by Lemma~\ref{lem:F1t1} the graph $G_1$ admits an implicit representation. 

For any $i>1$, the $G_i:=G[V_i\cup V_{i+1}]$ does not contain a one-sided copy of $F^1_{t,1}$ with the vertex of largest degree in $V_i$ (for odd $i$)
or with the vertex of largest degree in $V_{i+1}$ (for even $i$), since  otherwise together with $v$ this copy would induce a one-sided copy of $F_{t,1}$ with the vertex of largest degree in $U$. 
Therefore, by Lemma~\ref{lem:F1t1} the graph $G_i$ admits an implicit representation for all $i>1$.
Together with Theorem~\ref{thm:local} this implies an implicit representation for $G$.  
\end{proof}

For general one-sided forbidden $F_{t,p}$ the question remains open. Moreover, it remains open even for one-sided forbidden $2P_3$.
It is interesting to note that if we forbid $2P_3$ with black centers and if all black vertices have incomparable neighbourhoods, 
then the graph has bounded clique-width \cite{sperner} and hence admits an implicit representation. However, 
in general the clique-width of $2P_3$-free bipartite graphs is unbounded and the question of implicit representation  for one-sided forbidden $2P_3$ remains open.
 
%%%%%%%%%%%%%%%%%%%%%%%%%%%%%%%%%%%%%%%%%%%%%%%%%%%
\subsection{Subclasses of chordal bipartite graphs}
%%%%%%%%%%%%%%%%%%%%%%%%%%%%%%%%%%%%%%%%%%%%%%%%%%%

%A bipartite graph is {\it chordal} bipartite if it has no chordless (induced) cycles of length at least $6$. 
%It is known that the class of chordal bipartite graphs is superfactorial \cite{Spi95} and hence does not admit an implicit representation.
%On the other hand, 
Any subclass of chordal bipartite graphs excluding a forest is factorial, as was shown in \cite{chordal-bipartite}.
However, implicit representations for such subclasses are in general unavailable. Below we provide an implicit representation for  the class of $S_{2,2,2}$-free chordal bipartite graphs, 
which recently attracted attention in a different context \cite{wl}. We emphasize that the class of $S_{2,2,2}$-free bipartite graphs (without the restriction to {\it chordal} bipartite graphs) is 
superfactorial and  hence does not admit an implicit representation.

\begin{theorem}
The  class of $S_{2,2,2}$-free chordal bipartite graphs admits an implicit representation.
\end{theorem}

\begin{proof}
Similarly to Lemma~\ref{lem:ftt-1} we consider a connected graph $G$ in the class and a vertex of maximum degree $v$ in $G$. 
We also denote by $V_i$ the vertices of distance $i$ from $v$ and show that for any $i$, the bipartite graph $G_i:=G[V_i\cup V_{i+1}]$
belongs to a class that admits an implicit representation. First, we observe that $G$ (and hence each $G_i$) is $C_6$-free (since it is chordal),
and that $C_6\in \widetilde{\mathcal{M}}$, since the bipartite complement of $C_6$ is $3K_2$. It remains to show that each $G_i$ is $pK_2$-free for some constant $p$, 
implying that $G_i$ has bounded chain partition number (Theorem~\ref{thm:cpn}) and hence admits an implicit representation (Theorem~\ref{thm:local}).

The fact that $G_1$ is $3K_2$-free is obvious, since otherwise an induced $S_{2,2,2}$ can be easily found. 
Now we show that $G_i$ is $3K_2$-free for all $i$. Assume the contrary: there is a minimum $i>1$ such that 
$G_i$ contains a $3K_2$ induced by vertices $a_1,a_2,a_3\in V_i$ and $b_1,b_2,b_3\in V_{i+1}$ with $a_jb_j\in E(G)$ for all $j$.
Then at least two of $a_1,a_2,a_3$ have a common neighbour in $V_{i-1}$, since otherwise an induced $3K_2$ arises in $G_{i-1}$, contradicting the minimality of $i$. 
Without loss of generality assume $a_1$ and $a_2$ are adjacent to a vertex $c\in V_{i-1}$. To avoid an induced $S_{2,2,2}$, vertex $c$ is not adjacent to $a_3$
and $i\le 2$ (since otherwise we may consider a neighbour $d\in V_{i-2}$ of $c$ and $e\in V_{i-3}$ of $d$). We must thus have $i=2$. 

Now consider a neighbour $d\in V_1$ of $a_3$.
If $d$ has 2 or 0 neighbours in $\{a_1,a_2\}$, then an induced $S_{2,2,2}$ can be easily found (with centres $d$ or $c$ respectively). So, assume $d$ is adjacent to $a_2$ and non-adjacent to $a_1$. 
Then $v$ has a neighbour $e$ non-adjacent to $a_2$, since otherwise the degree of $a_2$ is greater than the degree of $v$. If $e$ is not adjacent to $a_1$,
then an induced $S_{2,2,2}$ with centre $c$ arises. Similarly, if it is not adjacent to $a_3$, then an induced $S_{2, 2, 2}$ with centre $d$ arises. 
However, and if $e$ is adjacent to both $a_1$ and $a_3$, then vertices $a_1,a_2,a_3,c,d,e$ induce a $C_6$, which is forbidden. 
A contradiction in all cases shows that $G_i$ is $3K_2$-free for all $i$ and completes the proof.  
\end{proof}

We observe that the class of $S_{2,2,2}$-free chordal bipartite graphs extends the class of bipartite permutation graphs, which has bounded symmetric difference by Theorem~\ref{thm:symdif}. 
One may ask whether this result can be extended to $S_{2,2,2}$-free chordal bipartite graphs; 
the answer is negative: it is possible to construct $S_{2,2,2}$-free chordal bipartite graphs with arbitrarily large symmetric difference. 
For brevity, we omit the full construction. We mention, however, that our examples have the form $G = (A, B \cup C, E)$, where $G[A \cup B]$ and $G[A \cup C]$ are chain graphs. 
We call such graphs {\em linked chain graphs} (and one can show that linked chain graphs are $S_{2, 2, 2}$-free chordal bipartite).

%\textcolor{red}{Alternatively, we can include the construction here if you think it is worth it.}

%
%\begin{problem}
%Is symmetric difference bounded in the class of $S_{2,2,2}$-free chordal bipartite graphs?
%\end{problem}

\medskip 
Several factorial subclasses of chordal bipartite graphs defined by  forbidding a unicyclic graph (i.e.\ a graph containing a single cycle) have been identified in \cite{chordal-factorial}.
In particular, a factorial upper bound was shown for $Q$-free chordal bipartite and $A$-free chordal bipartite graphs (see Figure~\ref{fig:Q} for the graphs $Q$ and $A$). 
\begin{figure}[ht]
\begin{center} 
\begin{picture}(120,100)
\put(60,20){\circle*{3}}
\put(60,60){\circle*{3}}
\put(20,100){\circle*{3}}
\put(20,60){\circle*{3}}
\put(105,60){\circle*{3}}
\put(60,100){\circle*{3}}
\put(60,20){\line(0,1){40}}
\put(20,60){\line(1,0){40}}
\put(60,60){\line(0,1){40}}
\put(20,60){\line(0,1){40}}
\put(60,60){\line(1,0){45}}
\put(20,100){\line(1,0){40}}
\end{picture}
\begin{picture}(120,100)
\put(40,20){\circle*{3}}
\put(40,60){\circle*{3}}
\put(40,100){\circle*{3}}
\put(85,20){\circle*{3}}
\put(85,60){\circle*{3}}
\put(85,100){\circle*{3}}
\put(40,20){\line(0,1){40}}
\put(85,20){\line(0,1){40}}
\put(40,60){\line(0,1){40}}
\put(85,60){\line(0,1){40}}
\put(40,60){\line(1,0){45}}
\put(40,100){\line(1,0){45}}
\end{picture}\end{center}
\caption{The graphs $Q$ (left) and $A$ (right)}
\label{fig:Q}
\end{figure}
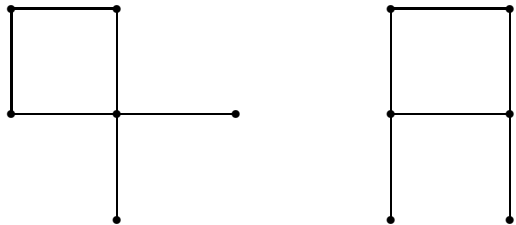

Below we strengthen these results in two ways. First, we extend both of them to the class of $D_k$-free chordal bipartite graphs, where a $D_k$ is the graph obtained from a cycle $C_4=(v_1,v_2,v_3,v_4)$
by adding one pendant edge to $v_1$, one pendant edge to $v_2$ and $k$ pendant edges to $v_4$. Second, we show that $D_k$-free chordal bipartite graphs admit an implicit 
representation, which is a stronger statement than a factorial upper bound on the size of the class. We observe that the class of $D_k$-free bipartite graphs (without the restriction to {\it chordal} bipartite graphs) is 
superfactorial and  hence does not admit an implicit representation. In our proof, we make use of the fact that any chordal bipartite graph has a vertex
which is {\it not} the centre of a $P_5$ \cite{Farber}.

\begin{theorem}\label{thm:dk-free}
The  class of $D_k$-free chordal bipartite graphs admits an implicit representation.
\end{theorem}

\begin{proof}
We consider a connected graph $G$ in the class and a vertex  $v$ in $G$ which is not the centre of a $P_5$. 
We denote by $V_i$ the vertices of distance $i$ from $v$ and show that for any $i$, the bipartite graph $G_i:=G[V_i\cup V_{i+1}]$
belongs to a class that admits an implicit representation.

Since $v$ is not the centre of a $P_5$, the graph $G_1$ is a chain graph, and hence admits an implicit representation.
For $i>1$, we show that the graph $G_i$ does not contain a one-sided copy of $F^1_{1,k}$ with the vertex of large degree in $V_i$. 
Indeed, assume that $G_i$ contains a one-sided copy of $F^1_{1,k}$ with the vertex of large degree in $V_i$,
and denote the two vertices of this copy in $V_i$ by $a$ and $b$.  
By definition, $a$ and $b$ have neighbours in $V_{i-1}$.
If they have a common neighbour $c$ in $V_{i-1}$, 
then the copy of $F^1_{1,k}$ together with $c$ and any neighbour of $c$ in $V_{i-2}$ induce a $D_k$ (where $V_0=\{v\}$). 
If $a$ and $b$ have no common neighbours in $V_{i-1}$,
then an induced cycle of length at least $6$ can be easily found, which is forbidden for chordal bipartite graphs. A contradiction in both cases shows that  
$G_i$ does not contain a one-sided copy of $F^1_{1,k}$ and hence admits an implicit representation by Lemma~\ref{lem:F1t1}.
Therefore, by Theorem~\ref{thm:local}, $G$ admits an implicit representation as well.
\end{proof}

Noting that $Q$ and $A$ are chain graphs, we next provide a different generalisation of these results 
by showing that the class of chordal bipartite graphs avoiding a chain graph admits an implicit representation. 
Chain graphs have a well-known universal construction \cite{universal}. More specifically, any chain graph on 
$n$ vertices is induced in the universal chain graph $Z_n$ on $2n$ vertices with parts $a_1, \dots, a_n$ and 
$b_1, \dots, b_n$, and with $a_i$ adjacent to $b_j$ whenever $j \geq i$ (see Figure~\ref{fig:chain} for an illustration). 
It thus suffices to show that, for fixed $k$, the class of $Z_k$-free chordal bipartite graphs admits an implicit representation.

\begin{figure}[ht]
	\begin{center}
		\begin{tikzpicture}[scale=1, transform shape]
			
			\foreach \i in {1,...,5} {
				\filldraw (\i * 2, 0) circle (2pt) node[below right]{$b_{\i}$};
				\filldraw (\i * 2, 2) circle (2pt) node[above left]{$a_{\i}$};
				\foreach \x in {\i,...,5} {
					\draw (\i * 2, 2) -- (\x * 2, 0);
				}
			}

		\end{tikzpicture}
	\end{center}
	\caption{The graph $Z_5$}
	\label{fig:chain}
\end{figure}
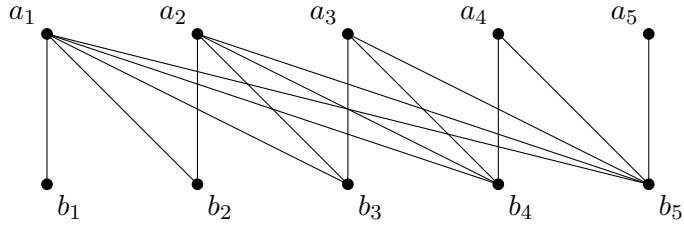

\begin{theorem}
	The class of chordal bipartite graphs avoiding a fixed chain graph admits an implicit representation.
\end{theorem}

\begin{proof}
As discussed above, it suffices to show the claim for $Z_k$-free chordal bipartite graphs. 
We prove this by induction on $k$. This is clear when $k = 1$, since those graphs are edgeless, 
and when $k = 2$, since those graphs are cographs and hence have bounded clique-width. In general, assume the statement is true for some $k \geq 2$, 
and consider the class of $Z_{k + 1}$-free chordal bipartite graphs. As in the proof of Theorem~\ref{thm:dk-free}, 
we find a vertex which is not the centre of a $P_5$, and define the graphs $G_i$ in the same way. 
Once more, $G_1$ is a chain graph. For $i > 1$, we show that the graph $G_i$ contains no induced copy of $Z_k$. 
Suppose for a contradiction that there was such a copy, labelled as in Figure~\ref{fig:chain}, 
with the $a$ vertices in $V_i$ and the $b$ vertices in $V_{i + 1}$. We then note that, by chordality, 
the graph induced by the $a$ vertices together with their neighbours in $V_{i - 1}$ is a chain graph (otherwise a $2K_2$ together with $b_k$ and a shortest path between the $2K_2$'s going through the $V_j$ with $j < i - 1$ would induce a large cycle). Since every $a$ vertex has a neighbour in $V_{i - 1}$, 
it follows that there must be a vertex $b_{k + 1}$ in $V_{i - 1}$ adjacent to all $a$ vertices. 
Together with a vertex $a_{k + 1}$ from $V_{i - 2}$ adjacent to $b_{k + 1}$, we obtain an induced $Z_{k + 1}$, 
which is the desired contradiction. As before, from Theorem~\ref{thm:local} and using the induction hypothesis, we are done. 
\end{proof}

%%%%%%%%%%%%%%%%%%%%%%%%%%%%%%%%%%%%%%%%%%%%
%%%%%%%%%%%%%%%%%%%%%%%%%%%%%%%%%%%%%%%%%%%%

\section{Factorial properties}
\label{sec:factorial}

%%%%%%%%%%%%%%%%%%%%%%%%%%%%%%%%%%%%%%%%%%%%
%%%%%%%%%%%%%%%%%%%%%%%%%%%%%%%%%%%%%%%%%%%%

A factorial speed of growth, as we mentioned in the introduction, is a necessary condition for a hereditary class to admit an implicit representation,
and hence determining the speed is the first natural step towards identifying new classes that admit such a representation. In this section, we 
prove a number of results related to the speed of some hereditary classes of bipartite graphs.

\subsection{Hypercubes}

We repeat that bounded functionality implies at most factorial speed of growth. 
Whether the reverse implication is also valid was left as an open question in \cite{functionality}.
It turns out that the answer to this question is negative. 
This is witnessed by the class $\mathcal{Q}$ of induced subgraphs of hypercubes.
Indeed, in \cite{functionality} it was shown that $\mathcal{Q}$ has unbounded functionality.
On the other hand, it was shown in \cite{Har20} that the class admits an implicit representation and is, 
in particular, factorial; in fact, more generally, 
the hereditary closure of Cartesian products of any finite set of graphs \cite{HWZ21} and even of any class admitting an implicit representation \cite{cartesian}, admits an implicit representation.
These results, however, are non-constructive and they provide neither explicit labeling schemes, nor specific factorial bounds on the number of graphs.
Below we give a concrete bound on the speed of $\mathcal{Q}$.

\begin{theorem}
There are at most $n^{2n}$ $n$-vertex graphs in $\mathcal{Q}$.
\end{theorem}

\begin{proof}
	Let $Q_n$ denote the $n$-dimensional hypercube, i.e.\ the graph with vertex set $\{ 0, 1 \}^n$, in which two vertices 
	are adjacent if and only if they differ in exactly one coordinate.
	%Let $\mathcal{Q}$ be the hereditary closure of the hypercubes $Q_n$. 
	%It is not hard to see that $\mathcal Q$ contains the class of trees, and so it has at least factorial speed of growth. It remains to show that it has at most factorial speed of growth. For this, 
	To obtain the desired bound, we will produce, for each labelled $n$-vertex graph in $\mathcal Q$, a sequence of $2n$ numbers between $1$ and $n$ which allows us to retrieve the graph uniquely.

	As a preliminary, let $G \in \mathcal{Q}$ be a connected graph on $n$ vertices. 
By definition of $\mathcal Q$, $G$ embeds into $Q_m$ for some $m$. We claim that, in fact, $G$ embeds into $Q_{n-1}$. 
If $m < n$, this is clear. Otherwise, using an embedding into $Q_m$, each vertex of $G$ corresponds to an $m$-digit binary sequence. 
For two adjacent vertices, the sequences differ in exactly one position. From this, it follows inductively that the $n$ vertices of $G$ all agree in at least $m-(n-1)$ positions. 
The coordinates on which they agree can simply be removed; this produces an embedding of $G$ into $Q_{n-1}$. 
Additionally, by symmetry, if $G$ has a distinguished vertex $r$, we remark that we may find an embedding sending $r$ to $(0, 0, \dots, 0)$.
	
	We are now ready to describe our encoding. Let $G \in \mathcal{Q}$ be any labelled graph with vertex set $\{x_1, \dots, x_n\}$. We start by choosing, for each connected component $C$ of $G$:
	
	\begin{itemize}
		\item a spanning tree $T_C$ of $C$;
		\item a root $r_C$ of $T_C$;
		\item an embedding $\varphi_C$ of $T_C$ into $Q_{n - 1}$ sending $r_C$ to $(0, 0, \dots, 0)$.
	\end{itemize} 
	
	Write $C_i$ for the component of $x_i$. We define two functions $p, d: V(G) \to [n]$ as follows:
	\[
	p(x_i) = \begin{cases}
		i, & \text{if $x_i = r_{C_i}$;} \\
		j, & \text{if $x_i \neq r_{C_i}$, and $x_j$ is the parent of $x_i$ in $T_{C_i}$.}
		\end{cases}	
	 \] 
	
	\[
	d(x_i) = \begin{cases}
		1, & \text{if $x_i = r_{C_i}$;} \\
		j, & \text{if $x_i \neq r_{C_i}$, and $\varphi_{C_i}(x_i)$ and $\varphi_{C_i}(x_{p(x_i)})$ differ in coordinate $j$.}
	\end{cases}	
	\] 
	
	One easily checks that the above maps are well-defined; in particular, when $x_i$ is not a root, the embeddings of $x_i$ and of its parent do, indeed, differ in exactly one coordinate. 
The reader should also know that the value of $d$ on the roots is, in practice, irrelevant -- setting it to 1 is an arbitrary choice. 
	
	We now claim that $G$ can be restored from the sequence $p(x_1), d(x_1), \dots, p(x_n), d(x_n)$. 
To do so, we first note that this sequence allows us to easily determine the partition of $G$ into connected components. 
Moreover, for each connected component $C$, we may then determine its embedding $\varphi_C$ into $Q_{n - 1}$: 
$\varphi_C(r_C)$ is by assumption $(0, 0, \dots, 0)$; we may then identify its children using $p$, then compute their embeddings using $d$; we may then proceed inductively. 
This information allows us to determine the adjacency in $G$ as claimed, and the encoding uses $2n$ integers between $1$ and $n$ as required. 
\end{proof}

\noindent
We conjecture that a stronger bound holds.

\begin{conjecture}\label{conj:small-counter}
	There exists a constant $c$ such that the number of $n$-vertex graphs in $\mathcal{Q}$ is at most
	$c^n n!$, i.e. the class $\mathcal{Q}$ is \emph{small},
\end{conjecture}

\noindent
If Conjecture \ref{conj:small-counter} is true, the class $\mathcal{Q}$ would be an explicit counter-example to \emph{the small conjecture} \cite{twin}, which says that every small class has bounded twin-width. This conjecture is known to be false \cite{small-counter}, but, to the best of our knowledge, no explicit counter-example is available.

\subsection{Subclasses of chordal bipartite graphs} 
\label{sec:fac-cb}

A super-factorial lower bound for the number of labelled $n$-vertex chordal bipartite graphs was shown in \cite{Spi95}. 
This result was improved in \cite{chordal-factorial} by showing that the speed remains super-factorial 
for the class of $(2C_4,2C_4^+)$-free chordal bipartite graphs, where $2C_4^+$ is the graph
obtained from $2C_4$ by adding an edge between the two copies of $C_4$. In this section, we show that for every {\it proper} 
induced subgraph $H$ of $2C_4$ or of $2C_4^+$, the speed of $H$-free chordal bipartite graphs is factorial. There are precisely 
three maximal such subgraphs, which we denote by $X$, $Y$ and $Z$ (see Figure~\ref{fig:XYZ}).

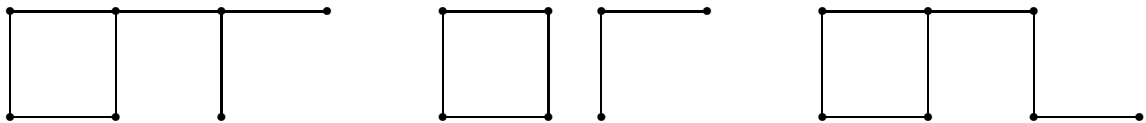
\begin{figure}[ht]
\begin{center} 
\begin{picture}(140,100)
\put(0,20){\circle*{3}}
\put(40,20){\circle*{3}}
\put(80,20){\circle*{3}}
\put(80,60){\circle*{3}}
\put(120,60){\circle*{3}}
\put(0,60){\circle*{3}}
\put(40,60){\circle*{3}}
\put(0,20){\line(1,0){40}}

\put(0,60){\line(1,0){40}}
\put(40,60){\line(1,0){40}}
\put(80,60){\line(1,0){40}}

\put(0,20){\line(0,1){40}}
\put(40,20){\line(0,1){40}}
\put(80,20){\line(0,1){40}}
%\put(50,0){Graph $X$}
\end{picture}
\begin{picture}(140,100)
\put(20,20){\circle*{3}}
\put(60,20){\circle*{3}}
\put(80,20){\circle*{3}}
\put(80,60){\circle*{3}}
\put(120,60){\circle*{3}}
\put(20,60){\circle*{3}}
\put(60,60){\circle*{3}}
\put(20,20){\line(1,0){40}}
\put(20,60){\line(1,0){40}}
\put(80,60){\line(1,0){40}}

\put(20,20){\line(0,1){40}}
\put(60,20){\line(0,1){40}}
\put(80,20){\line(0,1){40}}
\end{picture}
\begin{picture}(120,100)
\put(20,20){\circle*{3}}
\put(60,20){\circle*{3}}
\put(100,20){\circle*{3}}
\put(140,20){\circle*{3}}
\put(100,60){\circle*{3}}
%\put(140,60){\circle*{3}}
\put(20,60){\circle*{3}}
\put(60,60){\circle*{3}}
\put(20,20){\line(1,0){40}}
%\put(60,20){\line(1,0){40}}
\put(100,20){\line(1,0){40}}

\put(20,60){\line(1,0){40}}
\put(60,60){\line(1,0){40}}
%\put(100,60){\line(1,0){40}}

\put(20,20){\line(0,1){40}}
\put(60,20){\line(0,1){40}}
\put(100,20){\line(0,1){40}}
%\put(140,20){\line(0,1){40}}
\end{picture}
\end{center}
\caption{The graphs $X$ (left), $Y$ (middle),  and $Z$ (right)}
\label{fig:XYZ}
\end{figure}

In the proof, we use the following analogue of Theorem~\ref{thm:local} proved in \cite{local}.
\begin{theorem}{\rm \cite{local}}\label{thm:factorial}
Let $\mathcal X$ be a class of graphs and $c$ a constant. If every graph $G\in \mathcal X$ can be covered by graphs from 
a class $\mathcal Y$ of at most factorial speed of growth in such a way that every vertex of $G$ is covered by at most $c$ graphs, 
then $\mathcal X$ also has at most factorial speed of growth.
\end{theorem}

We also use the following result.

\begin{theorem}{\rm \cite{chordal-bipartite}}\label{thm:A}
	For any forest $F$, the class of chordal bipartite graphs excluding $F$ has at most factorial speed of growth.
%Any subclass of chordal bipartite graphs excluding a forest has at most factorial speed of growth.
\end{theorem}

%\begin{theorem}{\rm \cite{implicit}}\label{thm:B}
%If prime graphs in a hereditary class $X$ form a class with at most factorial speed of growth, then $X$ also has at most factorial speed of growth.
%\end{theorem}

We observe that it is unknown whether the last result can be strengthened by replacing ``at most factorial speed of growth'' with ``implicit representation''. 
In particular, it remains an open problem whether a factorial upper bound obtained for the three subclasses chordal bipartite graphs in the next three theorems can be 
strengthened to an implicit representation.  

\begin{theorem}\label{thm:X}
The class of $X$-free chordal bipartite graphs is factorial.
\end{theorem}

\begin{proof}
Let $G$ be a connected $X$-free chordal bipartite graph, let $v$ be a vertex of $G$, 
and let $V_i$ be the set of vertices of $G$ at distance $i$ from $v$. Assume $G[V_i \cup V_{i + 1}]$ ($i\ge 1$) contains 
an induced $P_{12}=(x_1,\ldots,x_{12})$ with even-indexed vertices in $V_i$ and odd-indexed vertices in $V_{i+1}$. 
To avoid an induced cycle of length at least $6$, vertices $x_2$ and $x_4$ must have a common neighbour $a\in V_{i-1}$.
Similarly, vertices $x_8$ and $x_{10}$ must have a common neighbour $b\in V_{i-1}$. If $b$ is adjacent to $x_2$, 
then vertices $x_1,x_2,x_3,x_8,x_9,x_{10},b$ induce an $X$ in $G$. Similarly, an induced $X$ arises if $b$ is adjacent to $x_4$,
and if $a$ is adjacent to $x_8$ or $x_{10}$.  
Therefore, $b$ is adjacent neither to $x_2$ nor to $x_4$, and $a$ is adjacent neither to $x_8$ nor to $x_{10}$, and hence $a\ne b$.
If additionally $a$ or $b$ is not adjacent to $x_6$, then an induced cycle of length at least $6$ can be easily found. 
Finally, if both $a$ and $b$ are adjacent to $x_6$, then vertices $x_2,x_3,x_4,x_6,x_7,a,b$ induce an $X$ in $G$.
A contradiction in all possible cases shows that $V_i$ and $V_{i+1}$ induce a $P_{12}$-free chordal bipartite graph for all $i\ge 1$.
By Theorems~\ref{thm:factorial} and~\ref{thm:A} this implies that the speed of  $X$-free chordal bipartite graphs is at most factorial.  
Since this class contains all chain graphs, its speed is at least factorial.
\end{proof}

\begin{theorem}\label{thm:Y}
The class of $Y$-free chordal bipartite graphs is factorial.
\end{theorem}

\begin{proof}
This class contains all chain graphs and hence its speed is at least factorial.
To prove that the speed is at most factorial, we will show that every graph $G$ in this class has either a vertex of bounded degree or a pair of vertices of bounded symmetric difference. 
If $G$ has a vertex $x$ of bounded degree, we create a record of the neighbours of $x$ and delete $x$. If the symmetric difference of two vertices $x$ and $y$ is bounded, 
we create a record containing vertex $y$ and the vertices in the symmetric difference of $x$ and $y$, and delete $x$. Applying this procedure recursively, we 
create a record of length $O(n\log n)$ (where $n=|V(G)|$), which allows us to restore the graph and shows that the class is at most factorial.  

Let $G$ be a $Y$-free chordal bipartite graph, let $v$ be a vertex of $G$ which is not the center of a $P_5$,  
and let $V_i$ be the set of vertices of $G$ at distance $i$ from $v$. 
Then $V_1$ and $V_2$ induce a $2K_2$-free bipartite graph, i.e.\ a chain graph.
If two vertices $a_i,a_j$ in $V_1$ have the same neighbourhood, then $\sd(a_i,a_j)=0$ and we are done. 
Therefore, we assume that all vertices in $V_1$ have pairwise different neighbourhoods. 
Using the notation of Figure~\ref{fig:chain}, we denote the vertices of $V_1$ by $a_1,\ldots,a_k$.
Also, for $i=1,\ldots,k-1$, let $B_i:=N(a_i)-N(a_{i+1})$. Note that for each $i$ the set $B_i$ is non-empty.

If $\deg(v)\le 4$, we are done, so assume $k\ge 5$.
Let $b$ be a vertex in $B_{k-4}$ and assume $b$ has at least $3$ neighbours $c_1,c_2,c_3$ in $V_3$.
Let $b'$ be a vertex in $B_{k-2}$. Then  
$b'$ has either two neighbours or two non-neighbours among $c_1,c_2,c_3$.
If $b'$ is adjacent to $c_1$ and $c_2$, then $b,b',c_1,c_2,v,a_{k-1},a_k$ induce a $Y$.
If $b'$ is not adjacent to $c_1$ and $c_2$, then $b,b',c_1,c_2,v,a_{k-3},a_{k-2}$ induce a $Y$.
A contradiction in both cases shows that $b$ has at most $2$ neighbours in $V_3$ and hence $\sd(v,b)\le 6$.
 \end{proof}

\begin{theorem}\label{thm:Z}
The class of $Z$-free chordal bipartite graphs is factorial.
\end{theorem}

\begin{proof}
Let $G$ be a connected $Z$-free chordal bipartite graph given together with a bipartition of its vertices into an independent set of white vertices and an independent set of black vertices. 
We will show that $G$ either contains no $P_{14}$, or it has two vertices of symmetric difference at most $2$. 
If $G$ is $P_{14}$-free, then $G$ belongs to a factorial class by Theorem~\ref{thm:A}. 
Using this, it is then routine to produce a (not necessarily implicit) representation of $G$ using $O(n\log n)$ bits 
by iteratively removing vertices of low symmetric difference until we are left with a $P_{14}$-free graph.
    
To show that $G$ either contains no $P_{14}$, or it has two vertices of symmetric difference at most $2$, 
we assume that $G$ contains a $P_{14}$, and extend it to a maximal induced tree that we denote by $T$. 
We claim that the vertices in $T$ are not distinguished by the vertices in $G - T$. 
This immediately yields our two vertices of low symmetric difference: if $T$ has at least three leaves, 
then it has two in the same side of the bipartition, and the only two vertices possibly distinguishing them 
are their neighbours in $T$. If $T$ has only two leaves, it is a path, and we can easily find two vertices on the path 
with symmetric difference at most $2$ (even at most $1$). Therefore, it suffices to prove the following claim.
    
    \begin{claim*}
        Let $T$ be as above, and let $x \in G - T$. Then $N_T(x)$ is either empty, or consists of all vertices of $T$ lying in one part of the bipartition.  
    \end{claim*}
    
    \begin{proof}[Proof of claim] \let\qed\relax
        Suppose $x$ has, without loss of generality, a white neighbour in $T$, and let $D$ be the set of all white neighbours of $x$ in $T$. We prove the claim in a series of steps.
        
        \begin{itemize}
            \item[i.] $|D| \geq 2$, since otherwise the tree $T$ is not maximal.
            
            \item[ii.] If $w, w' \in D$, then $w'' \in D$ for each white vertex $w''$ lying on the path in $T$ between $w$ and $w'$, since otherwise an induced cycle of length at least 6 arises. 
        \end{itemize}
        
        From the above, we immediately obtain:
        
        \begin{itemize}
            \item[iii.] Every vertex of $D$ belongs to a $C_4$ in $T \cup \{x\}$.
        \end{itemize}
        
        We then note:
        
        \begin{itemize}
            \item[iv.] No vertex of $T - D$ lies at distance 4 or more from $D$, since otherwise an induced $Z$ arises.
        \end{itemize}
        
        In particular, any white vertex is at distance at most 2 from $D$. Since $T$ contains a $P_{14}$ (and hence has two white vertices at distance 12), the triangle inequality implies:
        
        \begin{itemize}
            \item[v.] There exist two vertices in $D$ at distance at least 8. 
        \end{itemize}
        
        Together with ii., this in turn implies: 
        
        \begin{itemize}
            \item[vi.] There exists an induced path $P~=~(v_1, \dots, v_9)$ in $T$ (with edges $v_iv_{i + 1}$) such that $v_1, v_3, v_5, v_7, v_9 \in D$. 
        \end{itemize}

Let $w_0$ be a white vertex in $T - D$ closest to $P$ in $T$, and let $Q=(w_0,w_1,\ldots,w_k)$ be the unique path from $w_0$ to $P$ in $T$
(with $w_k\in V(P)$). If $k=1$, say $w_k=v_4$ (without loss of generality), then $w_{0},w_1,v_5,x,v_7,v_8,v_9$ induce a $Z$.
If $k\ge 2$, then $x$ is adjacent to $w_2$ (due to the choice of $w_0$), and assuming, without loss of generality, that $w_2$ is different from $v_7$ and $v_9$,
we conclude that $w_{0},w_{1},w_2,x,v_7,v_8,v_9$ induce a $Z$.  
A contradiction in both cases shows that $T - D$ contains no white vertices, thus proving the claim and the theorem. \end{proof}
\end{proof}

%%%%%%%%%%%%%%%%%%%%%%%%%%%%%%%%%%%%%%%%%%%%
%%%%%%%%%%%%%%%%%%%%%%%%%%%%%%%%%%%%%%%%%%%%

\section{Conclusion}
\label{sec:con}

%%%%%%%%%%%%%%%%%%%%%%%%%%%%%%%%%%%%%%%%%%%%
%%%%%%%%%%%%%%%%%%%%%%%%%%%%%%%%%%%%%%%%%%%%

In this paper, we proved several results related to graph parameters, implicit representation and factorial properties and raised a number of open questions.
In particular, we asked (in the form of conjectures) whether bounded symmetric difference implies implicit representation
and whether symmetric difference is bounded in the class of $P_7$-free bipartite graphs. Concerning implicit representations, one of the minimal classes 
for which this question is open is the class of $P_7$-free chordal bipartite graphs. It is also open for the three subclasses of chordal bipartite graphs 
from Section~\ref{sec:fac-cb}, for the class of bipartite graphs excluding a one-sided forbidden copy of an unbalanced $2P_3$. 
An explicit description of a labelling scheme that provides an implicit representation for induced subgraphs of hypercubes also remains an open problem.

\bigskip

\noindent
\textbf{Acknowledgement}.
We would like to thank the anonymous referee for many helpful suggestions, which improved the presentation of this paper.


\begin{thebibliography}{99}
\bibitem{IWOCA2022}
B. Alecu, V.E. Alekseev, A. Atminas, V. Lozin, V. Zamaraev, 
Graph parameters, implicit representations and factorial properties. 
{\it Lecture Notes in Computer Science}, 13270 (2022) 60--72.

\bibitem{functionality}
B. Alecu, A. Atminas, V. Lozin, 
Graph functionality.
{\it J. Combin. Theory Ser. B}, 147 (2021) 139--158.

\bibitem{cographs}
B. Alecu, V. Lozin, D. de Werra,
The micro-world of cographs.
{\it Discrete Appl. Math.}, 312 (2022) 3--14.

\bibitem{Allen}
P. Allen, Forbidden induced bipartite graphs. 
{\it J. Graph Theory} 60 (2009) 219--241.
		

\bibitem{Aistis}
A. Atminas, 
Classes of graphs without star forests and related graphs.
{\it Discrete Math.}, 345 (2022) 113089.
		
\bibitem{implicit}
A. Atminas, A. Collins, V. Lozin, and V. Zamaraev,
Implicit representations and factorial properties of graphs.
{\it Discrete Math.}, 338 (2015) 164--179.

\bibitem{SpHerProp} 
J. Balogh, B. Bollob\'{a}s, D. Weinreich, 
The speed of hereditary properties of graphs.
\textit{J. Combin. Theory Ser. B} 79 (2000) 131--156.

\bibitem{twin}
\'{E}. Bonnet, C. Geniet, E. J. Kim, S. Thomass\'e, R. Watrigant,
Twin-width II: small classes. 
{\it Combinatorial Theory}, 2 (2) (2022)  Paper No. 10, 42 pp.

\bibitem{small-counter}
\'{E}. Bonnet, C. Geniet, R. Tessera,  S. Thomass\'e,
Twin-width VII: groups. {\it arXiv preprint arXiv:2204.12330} (2022)

\bibitem{sperner}
E. Boros, V. Gurvich, M. Milanic, 
Characterizing and decomposing classes of threshold, split, and bipartite graphs via 1-Sperner hypergraphs. 
{\it J. Graph Theory} 94 (2020) 364--397.

\bibitem{branstadt-graph-classes}
A. Brandst\"{a}dt, V. B. Le, J. Spinrad,  Graph Classes: A Survey. 
SIAM Monographs on Discrete Mathematics and Applications (1999), xii+304 pp.

\bibitem{chordal-factorial}
K. Dabrowski, V.V. Lozin, and V. Zamaraev,
On factorial properties of chordal bipartite graphs.
{\it Discrete Math.}, 312 (2012) 2457--2465.

\bibitem{hindex}
D. Eppstein, E.S. Spiro, 
The $h$-index of a graph and its application to dynamic subgraph statistics.
{\it J. Graph Algorithms and Applications} 16 (2012) 543--567.

\bibitem{cartesian}
L. Esperet, N. Harms, and V. Zamaraev,
Optimal adjacency labels for subgraphs of Cartesian products.
{\it arXiv preprint arXiv:2206.02872} (2022).

\bibitem{Farber}
M. Farber, 
Characterizations of strongly chordal graphs.
{\it Discrete Math.}, 43 (1983) 173--189.

	\bibitem{contiguity}
P. Goldberg, M. Golumbic, H. Kaplan, R. Shamir, 
Four strikes against physical mapping of DNA.
{\it Journal of Computational Biology} 2 (1) (1995), 139--152.

\bibitem{Har20}
N. Harms, Universal Communication, Universal Graphs, and Graph Labeling.
{\it Proceedings of the 11th Innovations in Theoretical Computer Science Conference (ITCS 2020)} (2020).


\bibitem{HWZ21}
N. Harms, S. Wild, V. Zamaraev, 
Randomized communication and implicit graph representations.
{\it Proceedings of the 54th Annual ACM SIGACT Symposium on Theory of Computing (STOC 2022)} (2022), 1220--1233


\bibitem{false}
H. Hatami, P. Hatami, The implicit graph conjecture is false.
{\it IEEE 63rd Annual Symposium on Foundations of Computer Science (FOCS 2022)} (2022), 1134--1137

\bibitem{implicit-1}
S. Kannan, M. Naor, S. Rudich,
Implicit representation of graphs.
{\it SIAM J. Discrete Math.}, 5 (1992) 596--603.

%\bibitem{radius-2}
%H. A. Kierstead, S. G. Penrice, Radius two trees specify $\chi$-bounded classes. {\em J. Graph Theory} 18(2) (1994) 119--129.

% \bibitem{girth}
% F. Lazebnik, V. A. Ustimenko, and A. J.Woldar, A new series of dense graphs
% of high girth. {\it Bull AMS} 32 (1995), 73--79.

\bibitem{skew-star}
V. Lozin, Bipartite graphs without a skew star.
{\it Discrete Math.} 257 (2002) 83--100.

%\bibitem{factorial}
%V. Lozin, C. Mayhill, V. Zamaraev,
%A note on the speed of hereditary properties,
%{\it Electron. J. Combin.} 18 (2011) Research paper 157.

\bibitem{local}
V. Lozin, C. Mayhill, V. Zamaraev,
Locally bounded coverings and factorial properties of graphs.
{\it European J. Combinatorics}, 33 (2012) 534--543.

\bibitem{universal}
V. Lozin, G. Rudolf,
Minimal universal bipartite graphs.
{\it Ars Combin.}  84  (2007) 345--356.

\bibitem{chordal-bipartite}
V. Lozin and V. Zamaraev,
Boundary properties of factorial classes of graphs.
{\it J. Graph Theory}, 78 (2015) 207--218.


\bibitem{P7}
V. Lozin and V. Zamaraev,
The structure and the number of $P_7$-free bipartite graphs.
{\it European J. Combinatorics}, 65 (2017) 143--153.

\bibitem{wl}
I. Ponomarenko, G. Ryabov, 
The Weisfeiler-Leman dimension of chordal bipartite graphs without bipartite claw.
{\it Graphs Combin.} 37 (2021), 1089--1102.

%\bibitem{short}
%E.R.~Scheinerman,
%Local representations using very short labels,
%{\it Discrete Mathematics} 203 (1999) 287--290.

\bibitem{Spi95}
J.P.~Spinrad,
 Nonredundant 1's in $\Gamma$-free matrices,
{\em SIAM J.~Discrete Math.} 8 (1995) 251--257.

\bibitem{Spinrad}
J.P. Spinrad,  Efficient graph representations. 
Fields Institute Monographs, 19. American Mathematical Society, Providence, RI, 2003. xiii+342 pp.


\end{thebibliography}
\end{document}